\documentclass[11pt]{article}
\usepackage[utf8]{inputenc}
\usepackage[pagebackref=true]{hyperref}
\usepackage{amssymb,enumitem,amsmath,bm,bbm,amsthm,microtype,geometry,xcolor,tocbibind}
\definecolor{DarkBlue}{rgb}{0.15,0.15,0.55}
\hypersetup{linktocpage=true,colorlinks=true,allcolors=DarkBlue}
\usepackage[capitalise]{cleveref}

\newcommand{\R}{\mathbb{R}}
\newcommand{\Z}{\mathbb{Z}}
\newcommand{\N}{\mathbb{N}}
\def\One{\mathbbm{1}} 
\newtheorem{proposition}{Proposition}[section]
\newtheorem{theorem}[proposition]{Theorem}
\newtheorem{lemma}[proposition]{Lemma}
\newtheorem{corollary}[proposition]{Corollary}
\theoremstyle{definition}
\newtheorem{definition}[proposition]{Definition}
\usepackage{tikz}

\newcommand{\eps}{\varepsilon}
\newcommand{\e}{\mathrm{e}}
\newcommand{\E}{\mathbb{E}} 
\renewcommand{\P}{\mathbb{P}} 
\newcommand{\Y}{\mathcal{Y}} 
\newcommand{\rd}{\,\mathrm{d}}
\DeclareMathOperator{\ent}{H} 
\DeclareMathOperator{\KL}{KL} 
\renewcommand{\H}{\mathcal{H}} 

\usepackage{mathtools}
\DeclarePairedDelimiter\abs{\lvert}{\rvert}
\DeclarePairedDelimiter\norm{\lVert}{\rVert}
\DeclarePairedDelimiter\of{\lparen}{\rparen}
\DeclarePairedDelimiter\prn{\lparen}{\rparen}
\DeclarePairedDelimiter\bra{[}{]}
\DeclarePairedDelimiter\floor{\lfloor}{\rfloor}

\providecommand\given{}
\newcommand\SetSymbol[1][]{%
\nonscript\:#1\vert
\allowbreak
\nonscript\:
\mathopen{}}
\DeclarePairedDelimiterX\set[1]\{\}{%
\renewcommand\given{\SetSymbol[\delimsize]}
#1
}

\newtheoremstyle{remark}
{0.6em}
{0.6em}
{}
{}
{\itshape}
{.}
{.5em}
{}
\theoremstyle{remark}
\newtheorem*{remark}{Remark}
\newtheorem*{example}{Example}

\newenvironment{enum}{\begin{enumerate}

\setlength{\itemsep}{0pt}}{\end{enumerate}}

\title{Entropic Compressibility of Lévy Processes}
\author{Julien Fageot\thanks{AudioVisual Communication Lab (LCAV), École Polytechnique Fédérale de Lausanne (EPFL), Lausanne, Switzerland
	(see \url{http://bigwww.epfl.ch/fageot/index.html}).}
        \and Alireza Fallah\thanks{Laboratory for Information \& Decision Systems, Massachusetts Institute of Technology, Cambridge, MA, USA.}
		\and Thibaut Horel\footnotemark[2]}
\date{}

\begin{document}
%
%
%
%



\maketitle

\begin{abstract}
	\looseness=-1
	In contrast to their seemingly simple and shared structure of independence and stationarity, Lévy processes exhibit a wide variety of behaviors, from the self-similar Wiener process to piecewise-constant compound Poisson processes. Inspired by the recent paper of Ghourchian, Amini, and  Gohari~\cite{Ghourchian2017compressible}, we characterize their compressibility by studying the entropy of their double discretization (both in time and amplitude) in the regime of vanishing discretization steps. For a Lévy process with absolutely continuous marginals, this reduces to understanding the asymptotics of the differential entropy of its marginals at small times, {for which we obtain a new local central limit theorem}. We generalize known results for stable processes to the non-stable case, {with a special focus on Lévy processes that are locally self-similar}, and conceptualize a new compressibility hierarchy of Lévy processes, captured by their Blumenthal--Getoor index.
\end{abstract}


%

\section{Introduction}\label{sec:intro}
    
    Lévy processes generalize the Wiener process by relaxing the Gaussianity of increments, while retaining their stationarity and independence. They have proven to be a fruitful and flexible stochastic continuous-domain model in many applications, including financial mathematics~\cite{schoutens2003levy}, movement patterns of animals~\cite{humphries2010environmental}, turbulence~\cite{barndorff2004levy}, and sparse signal processing~\cite{Unser2014sparse}, among others.

	In this paper, we follow an entropy-based approach initiated in~\cite{Ghourchian2017compressible} to quantify the compressibility of a Lévy process, assuming that the process is approximated by uniform sampling in time and uniform quantization in amplitude. Specifically, let $L = (L_t)_{t\geq 0}$ be a Lévy process and for $n\geq 1$, consider the vector $(L_{1/n}, L_{2/n}, \ldots , L_{(n-1)/n}, L_1) \in \R^n$ obtained by sampling $L$ over the interval $[0,1]$ with a period of $1/n$.
    Since the increments of $L$ are independent and stationary, it is more convenient to work with the vector of independent and identically distributed increments $(\Delta_1,\dots,\Delta_n)$ where $\Delta_k = L_{k/n} - L_{(k-1)/n}$ for $k = 1, \ldots , n$ (and $L_0=0$ by definition). For an integer $m\geq 1$ and $x\in\R$, denote by
    \begin{displaymath}
        [x]_m = \frac 1 m \floor{m x},
    \end{displaymath}
    the quantization of $x$ with step size $1/m$. 
    Then, we define the \emph{entropy of the Lévy process $L$} for sampling period $1/n$ and quantization step $1/m$ as
    \begin{equation} \label{eq:defineentropy}
        \mathcal{H}_{n,m}(L) = 	\ent([\Delta_1]_m , \ldots, [\Delta_n]_m)
    \end{equation}
    where $\ent(X)$ denotes the entropy of the discrete random variable $X$.
    
    The main goal of this paper is to understand the asymptotic behavior of     \eqref{eq:defineentropy} when $n,m \to \infty$. 
    Due to the characterization of entropy in terms of compressibility, 
    a slower asymptotic growth of $\mathcal{H}_{n,m}(L)$
    implies that fewer bits are required to encode an approximation of $L$ of a given quality.
    Therefore, the asymptotic behavior of $\mathcal{H}_{n,m}(L)$ is a measure of the entropic compressibility of the Lévy process $L$. 

    \subsection{Contributions: the Asymptotic Entropy of Lévy Processes}
    \label{sec:contrib}
    
    Our work is directly inspired by the contribution of Hamid Ghourchian, Arash Amini, and Amin Gohari in~\cite{Ghourchian2017compressible}, who introduced the quantity $\mathcal{H}_{n,m}(L)$---with a slightly different but equivalent perspective---and characterized its asymptotic behavior for important classes of Lévy processes, mainly the compound Poisson and stable processes. 
    
    We provide a general analysis of the entropy $\mathcal{H}_{n,m}(L)$ of a Lévy process $L$, with a particular focus on Lévy processes whose marginals are absolutely continuous, thus allowing us to reduce this question to the study of the differential entropy $h(L_t)$ at small times $t\rightarrow 0^+$. 
    This reduction mostly benefits from the seminal contributions of~\cite{Ghourchian2017compressible}, to which we bring some useful complements. 
	Our main results, detailed below, extend the ones of~\cite{Ghourchian2017compressible} to all (possibly non-stable) Lévy processes whose marginals are absolutely continuous.

    \begin{itemize}[label=\raisebox{0.30ex}{\tiny$\bullet$}]

        \item \emph{The entropy of locally self-similar Lévy processes.}
         A Lévy process is locally symmetric and self-similar if its rescaled versions converge locally to a symmetric non-trivial random process~(see Definition~\ref{def:locallystable}). The rescaling depends on the local self-similarity order of the Lévy process, that is, the self-similarity order of its local limit, which is itself related to its Blumenthal--Getoor index $\beta$ (see Section \ref{sec:BGindex}). We show in Theorem \ref{theo:maintheolocallyselfsim} that the differential entropy $h(L_t)$ of a locally symmetric and self-similar Lévy process has the same asymptotic behavior as the symmetric-$\alpha$-stable process with $\alpha = \beta$; that is,
            \begin{equation}
                h(L_t) \underset{t \rightarrow 0^+}{\sim}  \frac{1}{\beta}\log t.
            \end{equation}
        
        \item \emph{An Upper-Bound for the Entropy of Lévy Processes.}
        For the general case, we obtain an upper-bound on the entropy in terms of the Blumenthal--Getoor index $\beta$ of the Lévy process. More precisely, we show that
        \begin{equation}
            \underset{t \rightarrow 0^+}{\lim \sup}\frac{h(L_t)}{\log (1/t)}\leq -\frac{1}{\beta}. 
        \end{equation}
        A direct consequence of this result is that for Lévy processes with $\beta=0$, the entropy diverges super-logarithmically as $t\to 0^+$, making them more compressible than any locally symmetric and self-similar Lévy process. 
       \end{itemize}
 
	   These two results together reveal a new entropy-based compressibility hierarchy of Lévy processes determined by their Bluementhal--Getoor index.
     Among Lévy processes with absolutely continuous marginals, the ones with Blumenthal--Getoor index $\beta=0$ are the most compressible, and more generally the smaller $\beta$, the more compressible $L$. We moreover exemplify our main contributions on several classes of Lévy processes.

    \subsection{Related Work}

	\paragraph*{Entropy \& Compressibility.}

	The notion of entropy introduced in 1948 by Claude Shannon in two successive groundbreaking papers~\cite{shannon2001mathematical, shannon2001mathematical2} is a fundamental measure of the quantity of information of discrete random sources and exactly captures their compressibility as formalized by the \emph{source coding theorem}~\cite[Theorem 5.4.1]{Cover2012elements}. 
    Initially defined for discrete random variables, it can be generalized to continuous random variables as the differential entropy, whose relation to discrete entropy via quantization was studied by Alfréd Rényi~\cite{renyi1959dimension}. Note however that contrary to the discrete entropy, there is no universal notion for quantifying the information of continuous random sources and the differential entropy is only one of several generalizations, with its strengths and weaknesses~\cite{jaynes1957information}. 
    
	Beyond discrete random vectors, one can consider the entropy of \emph{random sequences}, defined as the averaged limit of the vector case~\cite[Section 4.2]{Cover2012elements}.
	Going one step further, several generalizations to \emph{continuous-time} random processes have been proposed, as discussed extensively in the monograph of Shunsuke lhara~\cite{Ihara93information}.
    In this paper, we follow the recent work of Ghourchian et al.~\cite{Ghourchian2017compressible}, which provides a new definition of the entropy of a random process. Their approach is based on a double discretization in time (sampling) and amplitude (quantization), corresponding to \eqref{eq:defineentropy} for Lévy processes and relying on discrete entropy. While these authors consider general Lévy processes, whose marginals are not necessarily absolutely continuous, we note that among Lévy processes with absolutely continuous marginals, they only quantify the entropy of stable processes.

   \paragraph*{Sparsity \& Lévy Processes.}
    A signal of interest is sparse if it admits a concise representation that captures most of its information. Many naturally-occurring signals are deeply structured and admit such sparse representations, calling for sparse models and sparsity-promoting methods in the analysis and synthesis of signals. 
    This led for example to the introduction of $\ell_1$ methods in statistical learning~\cite{tibshirani1996regression,hastie2015statistical} which are ubiquitous in the field of compressed sensing
   ~\cite{Donoho2006,Candes2006sparse}. 

   Continuous-domain signals present their own challenges since adequate representations should ideally capture the fine details of possibly fractal-type sample paths~\cite{Mandelbrot1982fractal}.
   It is well-known that classic Gaussian models fail at modelling sparsity~\cite{Srivastava2003advances,Mumford2010pattern,Fageot2015wavelet}, in the sense that they generate poorly compressible data. Stable models with infnite variance~\cite{Taqqu1994stable,Pesquet2002stochastic} and more generally Lévy processes and Lévy fields~\cite{Unser2014sparse} have been proposed beyond Gaussian models.
   Those are particularly interesting since they include a wide range of random processes from non-sparse models, such as the Wiener process, to very sparse models, such as compound Poisson processes, whose rate of innovation is finite~\cite{Vetterli2002FRI}. 
    There is strong empirical evidence that Lévy processes are useful in modeling sparsity in signal processing~\cite{Unser2014sparse}. To the best of our knowledge, \cite{Ghourchian2017compressible} complemented by the present paper is the first attempt to provide an information-theoretic justification of this observation.
    
    \paragraph*{Approximation-theoretic Compressibility.}
    
    We also mention a distinct area of research which aims at quantifying the sparsity of random models, with a different perspective based on the theory of approximation. 
	It was initiated with the study of independent and identically distributed random sequences~\cite{Cevher2009learning} whose compressibility is measured by the asymptotic behavior of truncated subsequences~\cite{Amini2011compressibility} and is strongly linked to the heavy-tailedness of the distribution of the sequence~\cite{Gribonval2012compressibility,Silva2012compressibility}.
    Extensions to stationary ergodic random sequences and beyond have also been proposed~\cite{Silva2015characterization,silva2022compressibility}. 
    
    A continuous-domain function is compressible if most of its information is captured by a few coefficients in an adequate dictionary. In this regard, wavelet representations are well-known for their excellent compression rate~\cite{Mallat1999}. A natural line of research, has therefore been to study the wavelet approximability  of a random model, understood as the convergence rate of its best approximation in wavelet bases.
   It was shown in a series of work~\cite{Fageot2017besov,Fageot2017multidimensional,aziznejad2018wavelet,aziznejad2020wavelet},
   that the approximation error of the best $n$-term approximation of a Lévy process in wavelet bases behaves asymptotically like $n^{- 1/\beta}$, where $\beta$ is the Blumenthal--Getoor of the process, and decays faster than any polynomial for $\beta = 0$.
   This shows a wavelet-based hierarchy of Lévy processes, from the less sparse (the Wiener process) to the sparsest (compound Poisson or Lévy process $\beta=0$)~\cite{fageot2020nterm}. Remarkably, this coincides with the information-theoretic compressibility hierarchy provided in this paper.

    \subsection{Outline}
    
	Section~\ref{sec:levynoiseentropy} introduces the family of Lévy processes and their Blumenthal--Getoor index.
    The entropy   of a Lévy process is rigorously defined in Section~\ref{sec:defineentropy}, where we also provide existence results.
    Section \ref{sec:betanonzero} contains the main results of the paper: we first characterize the asymptotic behavior of the entropy of locally symmetric and self-similar Lévy processes in Section \ref{sec:selfsimtheo}, and then deduce a general upper bound for the entropy of any Lévy process in Section \ref{sec:upperboundentropy}.
    We apply our theoretical findings to specific classes of Lévy processes by characterizing their small-time entropic behavior in Section \ref{sec:examples}. 
    Finally, a summary and discussion of our results is provided in Section~\ref{sec:discuss}.

\section{Lévy Processes and their Blumenthal--Getoor Index}
\label{sec:levynoiseentropy}

We introduce the family of Lévy processes in Section~\ref{sec:levynoiseitself} and the subfamily of symmetric stable processes in Section~\ref{sec:SalphaS}. As we shall see in \cref{sec:betanonzero,sec:examples}, the entropic compressibility of a Lévy process is captured by its Blumenthal--Getoor index, whose definition and main properties are recalled in Section~\ref{sec:BGindex}.

\subsection{Lévy Processes}
\label{sec:levynoiseitself}

Lévy processes are named after Paul Lévy, who popularized their study in 1937 \cite[Chapter VII]{levy1954theorie}\footnote{Lévy processes had been previously introduced by Bruno de Finetti~\cite{de1929sulle}. Paul Lévy generalized existing results to the class of \emph{additive} processes, that is, processes satisfying \cref{def:levy} with the exception of stationarity. We refer the interested reader to \cite{Mainardi2008origin} for a detailed historical exposition of this matter.}.  They generalize the Wiener process by relaxing the requirement that increments be Gaussian. A brief presentation is given thereafter; more details can be found in the classic monographs~\cite{Applebaum2009levy,Bertoin1998levy,Sato1994levy}.

\begin{definition}\label{def:levy}
A \emph{Lévy process} is a continuous-time random process $L = (L_t)_{t\geq 0}$ satisfying:
\begin{enumerate}
	\setlength{\itemsep}{0pt}
    \item $L_0 = 0$ almost surely (a.s.).
    \item \emph{Stationary increments:} for all $s,t \geq 0$,  $L_s$ and $L_{t+s} - L_{t}$  have the same law.
\item \emph{Independent increments:} for all $n \geq 1$ and all $0\leq t_0 \leq t_1 \leq t_2 \leq \ldots \leq t_n$, the random variables $\of[\big]{L_{t_k} - L_{t_{k-1}}}_{1\leq k\leq n}$ are mutually independent.
\item \emph{Sample paths regularity:} $t\mapsto L_t$ is a.s.\ right-continuous with left limits over $\R_{\geq 0}$.
\end{enumerate}
\end{definition}

Recall that a random variable $X$ is infinitely divisible if it can be decomposed as the sum of $n$ i.i.d.\ random variables for all $n \geq 1$~\cite{Sato1994levy}.
 Lévy processes are intimately linked to \emph{infinitely divisible random variables}~\cite[Section 7]{Sato1994levy}. In particular, the marginals $L_t$ are infinitely divisible for all $t\geq 0$.
Indeed, we can write for each $t>0$ and $n\geq 1$,
\begin{equation}\label{eq:levy-division}
    L_t = \sum_{k=1}^n \left( L_{\frac k n t} - L_{\frac {k-1} n t} \right),
\end{equation}
where $(L_{\frac{k}{n}t} - L_{\frac{k-1}{n}t})_{1\leq k\leq n}$ is i.i.d.\ since $L$ has independent and stationary increments.

Moreover, the law of a Lévy process is completely characterized by the infinitely divisible law of $L_1$, its marginal at time $t=1$.
The characteristic function $\Phi_{L_1}$ of $L_1$ can be expressed as $\Phi_{L_1}(\xi) = \exp( \Psi (\xi))$, for $\xi \in \R$, where $\Psi : \R \rightarrow  \mathbb{C}$, the characteristic exponent of $L$, is a continuous function admitting a Lévy--Khintchine representation~\cite[Theorem 8.1]{Sato1994levy} 
\begin{equation} \label{eq:LevyKhintchine}
    \Psi(\xi) = \mathrm{i} \mu \xi - \frac{\sigma^2 \xi^2}{2} + \int_{\R} \left(\mathrm{e}^{\mathrm{i} \xi t }- 1 - \mathrm{i} \xi t \One_{|t|\leq 1} \right) \nu ( \mathrm{d} t) 
\end{equation}
for any $\xi \in \R$, with $\mu \in \R$, $\sigma^2 \geq 0$, and $\nu$ is a Lévy measure. The latter means that $\nu$ is a non-negative measure on $\R$ such that $\nu(\{0\}) = 0$ and
$$\int_{|t| \leq 1} t^2  \mathrm{d}\nu (t) +  \int_{|t| > 1} \mathrm{d}\nu (t) < \infty.$$ 
We call $(\mu, \sigma^2 , \nu)$ the Lévy triplet of $L$, which uniquely determines $L$.
Then, for all $t \geq 0$, the characteristic function of the infinitely divisible random variable $L_t$ is
\begin{displaymath}
    \Phi_{L_t}(\xi) = \exp \left( t \Psi(\xi) \right), \qquad \forall \xi \in \R. 
\end{displaymath}

\subsection{Symmetric-\texorpdfstring{$\alpha$}{alpha}-Stable Lévy Processes}
\label{sec:SalphaS}

Among Lévy processes, we now define the subfamily of symmetric-$\alpha$-stable processes, which will play an important role in our study, as potential limits in law of rescaled Lévy processes (see Section \ref{sec:selfsim}). More information can be found in the classical monograph by Gennady Samorodnitsky and Murad Taqqu~\cite{Taqqu1994stable}. 

A random variable $X$ is stable if, for all $n\geq 1$, there exists $c_n>0$ and $d_n \in \R$ such that
\begin{equation} \label{eq:defstable}
    X_1 + \cdots + X_n \overset{(\mathcal{L})}{=} c_n X + d_n,
\end{equation}
where the $X_k$ are independent copies of $X$ and $\overset{(\mathcal{L})}{=}$ stands for equality in law. From the definition, we readily see that a stable random variable is infinitely divisible. The entire family of stable random variables is described by four parameters~\cite[Section 34]{gnedenko1954limit}. If $X$ is moreover symmetric (\emph{i.e.}, $X$ and $-X$ have the same law), then its characteristic function takes the form 
\begin{displaymath}
    \Phi_X(\xi) = \exp ( - \gamma \lvert \xi\rvert^\alpha), \quad \forall \xi \in \R,
\end{displaymath}
where $\gamma > 0$ and $\alpha \in  (0,2]$. In this case, Eq.~\eqref{eq:defstable} is satisfied with $c_n = n^{1/\alpha}$ and $d_n = 0$~\cite[Theorem 1, Section VI.1]{Feller2008introduction}. The parameter $\gamma$ is simply a scale parameter, while the parameter $\alpha$ deeply influences the properties of $X$. This justifies the terminology symmetric-$\alpha$-stable laws.
For instance, the $p$th moment $\mathbb{E}[ |X|^p ]$ of $X$ is finite if and only if $\alpha = 2$ (in which case $X$ is a Gaussian random variable) or $0 < p < \alpha < 2$~\cite[Proposition 1.2.16]{Taqqu1994stable}. 

For $\alpha \in (0,2]$, a symmetric-$\alpha$-stable Lévy process (or S$\alpha$S process) is a Lévy process $L$ for which $L_1$ is a S$\alpha$S random variable. The characteristic exponent of $L$ is thus given by
\begin{equation} \label{eq:psiforstable}
    \Psi(\xi) = - \gamma \lvert \xi \rvert^\alpha, \quad \forall \xi \in \R
\end{equation}
and, for each $t\geq 0$, $L_t$ is S$\alpha$S with characteristic function $\Phi_{L_t}(\xi) = \exp( - \gamma t \lvert \xi \rvert^\alpha)$. 

\subsection{The Blumenthal--Getoor Index of Lévy Processes}
\label{sec:BGindex}

    The Blumenthal--Getoor index was introduced by Robert M. Blumenthal and Ronald K. Getoor in 1961 to characterize the small-time behavior of Lévy processes~\cite{Blumenthal1961sample}. 
    Since then, it has been recognized as a key quantity to characterize the local Besov regularity~\cite{Schilling2000function,Schilling1998growth,aziznejad2018wavelet} or the variations \cite{Rosenbaum2009first} of the sample paths, the Hausdorff dimension of the image set~\cite{Jaffard1999multifractal,Durand2012multifractal}, moment estimates~\cite{Luschgy2008moment,Deng2015shift,Fageot2017multidimensional,kuhn2017levy}, the local self-similarity~\cite{fageot2019scaling}, or the local wavelet compressibility~\cite{fageot2020nterm} of Lévy processes and their generalizations. 
    We demonstrate in this paper that it also quantifies the asymptotic behavior of the entropy of Lévy processes.
    
    A Lévy process satisfies the \emph{sector condition} if there exists a constant $C>0$ such that 
    \begin{equation} \label{eq:sector}
    |\Im \Psi(\xi) |\leq C |\Re \Psi(\xi)|, \quad \forall \xi \in \R,
    \end{equation}
	with $\Psi$ its characteristic exponent, and where $\Re$ and $\Im$ respectively stand for the real part and imaginary part of complex numbers. This condition ensures that no drift is dominating the process. It typically excludes the pure drift (deterministic) process $L = (\mu t)_{t\geq 0}$ with $\mu \in \R \backslash\{0\}$, for which $\Psi(\xi) = \mathrm{i} \mu \xi$
	and is automatically satisfied by symmetric Lévy processes, {for which $L_1$ and $-L_1$ have the same law and whose characteristic exponent $\Psi$ is therefore real}. See~\cite{Bottcher2014levy} for additional discussions about the sector condition. We shall always assume that the sector condition is satisfied in this paper without further mention.
    
    \begin{definition}
 The \emph{Blumenthal--Getoor index} $\beta$ of a Lévy process $L$ with characteristic exponent $\Psi$ is defined by
    \begin{equation}
        \label{eq:BGindex}
		\beta = \inf \set*{p > 0,\;\; \frac{\abs{\Psi(\xi)}}{\abs\xi^p} \underset{\abs\xi\to\infty}{\longrightarrow} 0 }.
    \end{equation}
    \end{definition}

    The Blumenthal--Getoor index lies in $[0,2]$. This follows directly from the fact that there exists a constant $C > 0$ such that $|\Psi(\xi)|\leq C |\xi|^2$ for any $|\xi|\geq 1$~\cite[Proposition 2.4]{FageotThese}. The characteristic exponent of a S$\alpha$S process being given by \eqref{eq:psiforstable}, we deduce that its index is $\beta = \alpha$.
    The characteristic exponent of the Laplace process, for which $L_1$ is a Laplace random variable, is given by $\Psi(\xi) = - \log ( 1 + \xi^2 )$~\cite{Koltz2001laplace}. We therefore have $\beta = 0$ in this case.
    This is also the case for the gamma process (see Section~\ref{sec:gammalaplace}). 
    
\begin{remark}
	It comes as no surprise, {in light of the seminal work of Blumenthal and Getoor~\cite{Blumenthal1961sample}}, that the local behavior of a Lévy process $L$ is captured (via the index $\beta$) by the asymptotic behavior of its characteristic exponent. The latter is indeed the logarithm of the characteristic function of the law of $L_1$, whose local properties are known to be linked to asymptotic properties of its Fourier transform. 
\end{remark}

\section{Entropy of Lévy Processes}
\label{sec:defineentropy}

\subsection{Quantization and Entropy of Random Variables}

We briefly review the definitions of discrete and differential entropy. The former applies to random variables taking values in a countable space, and the latter applies to \emph{absolutely continuous} random variables, \emph{i.e.}, variables whose distribution is absolutely continuous with respect to the Lebesgue measure and thus admit a probability density function.

\begin{definition}
Let $Y$ be a discrete random variable taking values into a countable space $\Y$. The \emph{discrete entropy} of $Y$ is defined by
\begin{equation} \label{eq:entropy}
	\ent(Y) =  - \sum_{y\in\Y} \P(Y=y) \log \prn*{\P(Y=y)}
= - \E \bra*{ \log\prn[\big]{ \P(Y) } },
\end{equation}
with the usual convention $0\log (0) = 0$. Since the summand in \eqref{eq:entropy} is non-negative, $\ent(Y)$ is either $+\infty$ or a non-negative real. In the latter case, we say that $Y$ has \emph{finite} entropy.

Let $X$ be an absolutely continuous real-valued random variable with probability density function $p_X$. The \emph{differential entropy} of $X$ is
\begin{equation}\label{eq:diff-entropy}
h(X) = -\int_\R p_X(x) \log p_X(x) \rd x= -\E\bra[\big]{\log p_X(X)},
\end{equation}
which is well-defined in $[-\infty,\infty]$ provided that either the positive or negative part of the integrand in \eqref{eq:diff-entropy} is integrable. We say that the real random variable $X$ has \emph{finite differential entropy} when it is absolutely continuous and the integrand in \eqref{eq:diff-entropy} is absolutely integrable.
 \end{definition}

The quantization of order $m$ (or $m$-quantization) of a real $x \in \R$ is defined by
\begin{displaymath}
	[x]_m = \frac 1 m \floor{m x }.
\end{displaymath}  
That is, $[x]_m$ is the unique element of $\Z / m = \set{n/m\given n\in\Z}$ such that $[x]_m   \leq x < [x]_m + \frac 1 {m}$. For a real random variable $X$, the study of the entropy of its quantization $\ent([X]_m)$, and its relation to $h(X)$ (when it exists) was initiated by Alfréd Rényi in \cite{renyi1959dimension}. In particular, he established the following result (see also \cite[Corollary 1]{Ghourchian2017compressible}).

\begin{proposition}[{\cite[(11) and Thm. 1]{renyi1959dimension}}] \label{prop:renyi}
Let $X$ be a real random variable. If $\ent(\floor X)<\infty$, then $\ent([X]_m)<\infty$ for all $m\geq 1$. If furthermore $X$ has finite differential entropy, then
\begin{equation}\label{eq:renyi}
	\ent([X]_m) = \log m + h(X) + o_m(1)\,.
\end{equation}
\end{proposition}

\begin{remark}
Finiteness of the differential entropy $h(X)$ does not necessarily imply finiteness of $\ent(\floor X)$, as exemplified by Rényi in the remark following~\cite[Theorem 1]{renyi1959dimension}. 
\end{remark}

For an absolutely continuous random variable $X$, it is known that finiteness of the log-moment $\E[\log(1+\abs X)]$ implies that $h(X)<\infty$ (see \emph{e.g.} \cite[Proposition 1]{Rioul2011information}). This is proved by a direct application of Gibbs' inequality to the Kullback--Leibler divergence $\KL(X\|Y)$ for a Cauchy variable $Y$. A simple adaption of the proof shows that this condition also implies $\ent(\floor X)<\infty$. We thus obtain the following proposition, proved in Appendix \ref{app:existenceentropy}. 

\begin{proposition}\label{prop:entropy-criterion}
Let $X$ be a real random variable such that $\E[\log(1+\abs X)]<\infty$, then $\ent(\floor X)$ is finite. If moreover $X$ is absolutely continuous with bounded probability density function (\emph{e.g.}, if $\Phi_X$ is integrable) then $X$ has finite differential entropy and \eqref{eq:renyi} holds.
\end{proposition}

\subsection{Definition and Existence of the Entropy of Lévy Processes}

The notion of $m$-quantization is extended to finite-dimensional vectors by writing $[\bm{x}]_m = ( [x_1]_m, \ldots , [x_n]_m) \in (\Z/m)^n$ for $\bm{x} = (x_1, \ldots , x_n) \in \R^n$. Next, we give the definition of the entropy of a Lévy process. 

\begin{definition} \label{def:entropy} 
Let $L$ be a Lévy process and $n,m\geq 1$ be integers. First, let
\begin{displaymath} \label{eq:bmXn}
 \bm{X}_n (L) = (L_{1/n} , L_{2/n} - L_{1/n}, \ldots , L_{1} - L_{(n-1)/n})
\end{displaymath}
be the random vector of sampled values of $L$. Then, the \emph{entropy of $L$ with time-quantization $n$ and amplitude-quantization $m$} is defined as
\begin{equation} \label{eq:definitionHnm}
\H_{n,m} (L) = \ent \prn[\big]{ [ \bm{X}_n (L) ]_m }
= n \ent\prn*{ \bra*{ L_{1/n} }_m } \in [0, \infty],
\end{equation}
where the second equality is because the coordinates of $\bm X_n(L)$ are independent and distributed as $L_{1/n}$ by definition of a Lévy process.
\end{definition}

 \begin{remark}
		Although the presentation is slightly different, our definition is essentially equivalent to the one of Ghourchian et al. \cite[Eq. (6)]{Ghourchian2017compressible}. There, the authors consider the entropy of Lévy white noises, that are weak derivatives of Lévy processes~\cite{Dalang2015Levy}. Thus, the entropy of a Lévy white noise $W$ in the sense of \cite{Ghourchian2017compressible} is equal to the entropy of the corresponding Lévy process $L$ such that $L' = W$ in Definition \ref{def:entropy} up to the following minor difference. 
Ghourchian et al.\ define the quantization of a random variable $X$ as $ \frac 1 m \floor{m x + 1/2}$ instead of $ \frac 1 m \floor{m x}$~\cite[Definition 6]{Ghourchian2017compressible}. This is purely a matter of convention and will not affect the results; we prefer to follow the quantization convention of Rényi~\cite{renyi1959dimension}. 
\end{remark}

Intuitively, $\H_{n,m} (L)$ measures the expected length of the most efficient encoding of $L$ over $[0,1]$, after approximating it both temporally (via sampling) and in amplitude (via quantization). In
\cite{Ghourchian2017compressible}, the authors argue that the asymptotic behavior of $\H_{n,m}(L)$ as $n$ and $m \to \infty$ can be used as a measure of the compressibility of $L$. 

{We now provide some useful inequalities for the entropy of Lévy processes,
resulting in a characterization of Lévy processes with finite entropy. Some of
these inequalities are already known but restated here for the sake of
completeness.}
{\begin{proposition} \label{prop:anewone}
    Let $L$ be a Lévy process, then for all integers $n,m \geq 1$ we have
    \begin{align} \label{eq:renyieprovedit}
        0 \leq 
        \H_{n,1}(L)
        \leq 
        \H_{n,m}(L) 
        \leq 
        \H_{n,1}(L) + n \log m
    \end{align}
    and 
    \begin{equation} \label{eq:ghourchiandus}
		\H_{1,m}(L) - \log n 
        \leq \H_{n,m}(L) \leq 
		n \H_{1,m}(L) + n \log n .
    \end{equation}
    In particular, the three following statements are equivalent:
    \begin{itemize}
        \item $\exists n,m \geq 1, \  \H_{n,m}(L) < \infty$;
        \item $\forall n,m \geq 1, \  \H_{n,m}(L) < \infty $;
        \item $\ent(\floor{L_1}) < \infty$.
    \end{itemize}
\end{proposition}}

It follows from Proposition~\ref{prop:anewone} that the entropy $\H_{n,m}(L)$ of a Lévy process $L$ is always finite or always infinite depending on the finiteness of $ \ent(\floor{L_1})$. This justifies the following definition. 

\begin{definition} \label{def:finiteentropy}
We say that a Lévy process $L$ has \emph{finite entropy} if $ \ent(\floor{L_1}) < \infty$ and \emph{infinite entropy} otherwise.
\end{definition}

\begin{proof}[Proof of Proposition~\ref{prop:anewone}]
    {Rényi proved in~\cite[Eq. (11) and (13)]{renyi1959dimension} that, for a random variable $Y$ and any integer $m \geq 1$,
    \begin{displaymath}
		0 \leq \ent([Y]_1) \leq \ent([Y]_m) \leq \ent([Y]_1) + \log m.
    \end{displaymath}
    Applying this relation to $Y=L_{1/n}$ and using \eqref{eq:definitionHnm}, we obtain \eqref{eq:renyieprovedit}.} 
    
    {The infinitely divisible random variable $X =L_1$ can be decomposed  as a sum $X = X_1 + \cdots + X_n$ of $n$ i.i.d.\ random variables whose common law is the one of $L_{1/n}$.
    Let $y_i \in \R$ for  $i= 1 \ldots n$ and set $y = \sum_{i=1}^n y_i$. Then, we have that 
        \begin{displaymath}
		\sum_{i=1}^n \floor{y_i} - 1 \leq y -1 < \floor{y} \leq y
		< \sum_{i=1}^n \floor{y_i} + n\,, 
    \end{displaymath}
    which implies that $e := \floor{y} - \sum_{i=1}^n \floor{y_i} \in \{0, \ldots, n-1\}$. Applying this relation to $y_i = m X_i$ and $y = mX = m \sum_{i=1}^n X_i$, we deduce that
    \begin{displaymath}
        E_{n,m} := [X]_m - \sum_{1=1}^n [X_i]_m \in \left\{ 0, \frac{1}{m}, \ldots, \frac{n-1}{m} \right\}. 
    \end{displaymath}
    In particular, as a discrete random variable that can take values among a set of cardinal at most $n$, we have that $\ent(E_{n,m}) \leq \log n$~\cite[Theorem 2.6.4]{Cover2012elements}.}
    
   {We then remark that
    \begin{align*}
        \mathcal{H}_{1,m}(L)
        &= 
        \ent([X]_m)
       =
        \ent \left( \sum_{i=1}^n [X_i]_m + E_{n,m} \right)
		 \leq 
		\ent \left(  [X_1]_m ,\ldots ,  [X_n]_m , E_{n,m} \right) \\
		& \leq 
		\sum_{i=1}^n  \ent \left( [X_i]_m \right) + \ent(E_{n,m})
		= n \ent([L_{1/n}]_m) + \ent(E_{n,m})
		\leq   n \ent([L_{1/n}]_m) + \log n  \\
		&= \mathcal{H}_{n,m}(L) + \log n, 
    \end{align*}
    where we used the subadditivity of the discrete entropy~\cite[2.6.6]{Cover2012elements}. This shows the leftmost inequality  in \eqref{eq:ghourchiandus}. }
    
    {For the rightmost inequality, we use the argument of ~\cite[Section IV-B, Proof of item 2)]{Ghourchian2017compressible} that
	$\ent([L_{1/n}]_1) \leq \ent([L_1]_1 ) + \log n$ and adapt it to quantification $m \geq 1$. We
	have that
	\begin{align*}
	    \frac{\mathcal{H}_{n,m}(L)}{n}  &
	    = \ent([L_{1/n}]_m) 
	    = \ent([X_1]_m)
	    = \ent\left( \sum_{i=1}^n [X_i]_m \big| [X_2]_m , \ldots , [X_n]_m \right) \\
	    &\leq 
	    \ent\left( \sum_{i=1}^n [X_i]_m \right)  = 
		\ent\left( [X]_m - E_{n,m} \right)
	    \leq 
	    \ent([X]_m , E_{n,m}) \\
	    & \leq \ent([X]_m) + \ent(E_{n,m})
	   \leq 
	    \ent([L_1]_m) + \log n  = \mathcal{H}_{1,m}(L) + \log n,
	\end{align*}
	where we used again the subadditivity of the discrete entropy.  Then, the rightmost inequality in \eqref{eq:ghourchiandus} clearly follows.}

   {Finally,  \eqref{eq:renyieprovedit} and \eqref{eq:ghourchiandus} together easily imply that, for all $n,m \geq 1$, 
    \begin{small}   \begin{equation} \label{eq:summaryHnmH1}
		   \mathcal{H}_{1,1}(L) - \log n
        \leq \H_{n,m}(L) \leq 
		n \mathcal{H}_{1,1}(L) + n \log (nm), 
    \end{equation}
    \end{small}
   from wich the equivalent statements in Proposition~\ref{prop:anewone} follow since $\ent(\floor{L_1}) = \mathcal{H}_{1,1}(L)$.}
\end{proof}

We now limit ourselves to Lévy processes whose marginals have finite differential entropy. 
A direct corollary of \cref{prop:renyi} is the following result, implicitly stated in \cite{Ghourchian2017compressible}.

\begin{corollary}[implicit in {\cite{Ghourchian2017compressible}}]
\label{prop:firstone} 
Let $L = (L_t)_{t\geq 0}$ be a Lévy process {with finite entropy (see Definition~\ref{def:finiteentropy}).} 
If $L_t$ has finite differential entropy for all $t > 0$, then
\begin{equation} \label{eq:secondtrivial}
\H_{n,m} (L) = n \prn*{ \log m + h ( L_{1/n} ) + {\eps_{n,m}} }.
\end{equation}
{where $\eps_{n,m}$ 
vanishes when $m \to \infty$ for all $n \geq 1$.
In particular,  there exists a non-decreasing function $m:\N\to\N$ such that $\lim_{n\to\infty}\eps_{n,m(n)} = 0$. }
\end{corollary}

\begin{proof}[Proof of Corollary~\ref{prop:firstone}]
	{According to Proposition~\ref{prop:anewone}, $\ent(\floor{L_{1/n}}) < \infty$ for all $n \geq 1$. 
	Applying \cref{prop:renyi} for $n \geq 1$ fixed, we deduce that \begin{displaymath} 
	    \frac{\mathcal{H}_{n,m}(L) }{n} =  \ent( [L_{1/n}]_m ) = \log m + h(L_{1/n}) + \eps_{n,m}
	\end{displaymath}
	where $\eps_{n,m} \rightarrow 0$ when $m\rightarrow \infty$, hence \eqref{eq:secondtrivial}  follows.
	Fix any vanishing sequence $(u_n)_{n\geq 1}$ of non-negative real numbers. Then, for any $n \geq 1$, there exists $m(n)$ such that $|\eps_{n,m(n)}| \leq u_n$. Moreover, $m(n)$ can be chosen so that $m(n)$ is non-decreasing. Then, $|\eps_{n,m(n)}| \leq u_n$ vanishes when $n \to \infty$.}
\end{proof}

{Understanding the set of functions $n\mapsto m(n)$ such that $\lim_{n\to\infty}\eps_{n,m(n)}=0$ is crucial in understanding the asymptotic behavior of $\H_{n,m}(L)$ as $n$ and $m$ grow to infinity. In particular, since $h(L_{1/n})\to-\infty$ as $n\to\infty$ (see \Cref{prop:entropy-basic-properties} below) and the entropy of $L$ is nonnegative, it follows from \eqref{eq:secondtrivial} that $\lim_{n\to\infty}\eps_{n,m(n)}=+\infty$ whenever $\log m(n)\in o\big(h(L_{1/n})\big)$, in which case the expansion \eqref{eq:secondtrivial} is non-informative. 

In contrast, when\footnote{Recall that the notation $f(n)\in\omega\big(g(n)\big)$ expresses that $f(n)=g(n)\cdot h(n)$ for some function $h$ satisfying $h(n)\to\infty$ as $n\to\infty$.} $\log m(n)\in\omega\big(h(L_{1/n})\big)$ and $\lim_{n\to\infty} \eps_{n,m(n)}=0$, then \eqref{eq:secondtrivial} provides the first two terms of the asymptotic expansion of $\H_{n,m(n)}(L)$. More generally, the smallest rate of growth of $n\mapsto m(n)$ such that $\eps_{n,m(n)}\in o\big(\log m(n) + h(L_{1/n})\big)$ quantifies how the amplitude quantization's order needs to be adjusted to increasing time quantization's orders. This “critical” rate of growth depends on $L$ and needs to be determined on a case-by-case basis. It follows from our discussion that we expect this critical rate to be at least of the order of $\exp\big(-h(L_{1/n})\big)$ and in fact, \cite[Prop.\ 1]{Ghourchian2017compressible} shows that for S$\alpha$S Lévy processes for which $h(L_{1/n})\sim - \frac 1 \alpha \log n$, we have $\lim_{n\to\infty} \eps_{n,m(n)} = 0$ whenever $m(n)\in\omega(n^{1/\alpha})$.}

\begin{remark}
{In~\cite{renyi1959dimension}, Rényi introduced the  dimension and $d$-dimension associated to a random variable $X$ as follows. Assume that there exist $d \in [0,1]$ and $h \in \R$ such that
	$\ent([X]_m) = d \log m + h + o_m(1)$. Then, $d$ is called the \emph{dimension} and $h$ the \emph{$d$-dimension} of $X$.}
	
    {Similarly for a Lévy process $L$, we denote by $d_n$ and $h_n$ the dimension and $d$-dimension of $L_{1/n}$, assuming they exist, and call $(d_n)_{n\geq 1}$ and  $(h_n)_{n \geq 1}$ the dimension sequence and the $d$-dimension sequence of $L$. Corollary~\ref{prop:firstone} then shows that the sequences $(d_n)_{n \geq 1}$ and $(h_n)_{n\geq 1}$ exist as long as $L$ has finite entropy and marginals with finite differential entropy, in which case $d_n = 1$ and $h_n = h(L_{1/n})$ for all $n\geq 1$.}
    
    {It is worth noting that it is possible to consider Lévy processes whose marginals are not absolutely continuous and to study their dimension and $d$-dimension sequences. For instance, Ghourchian et al.\ considered compound Poisson processes $L$ with rate $\lambda > 0$ and jump distribution with finite differential entropy for which the dimension sequence of $L$ is $d_n = (1 - \mathrm{e}^{-\lambda/n}) < 1$ ~\cite[Proposition 2]{Ghourchian2017compressible}. Investigating further the dimension sequences of Lévy processes whose marginals are not absolutely continuous is an interesting question for future work.}
\end{remark}

{Combining \Cref{prop:entropy-criterion} with \Cref{prop:firstone} we immediately obtain the following corollary identifying a mild condition under which $L$ has finite entropy and \eqref{eq:secondtrivial} holds. The proof is provided in \Cref{app:existenceentropy}.}

\begin{corollary} \label{prop:existenceentropy}
	Let $L$ be a Lévy process such that $\E [\log( 1 + |L_1|)] < \infty$, then $L$ has finite entropy. If moreover the characteristic function $\Phi_{L_t}$ is integrable for all $t>0$, then $L_t$ has finite differential entropy for all $t>0$ and \eqref{eq:secondtrivial} holds.
\end{corollary}

\begin{remark}
Finiteness of the log-moment $\E[\log(1 + |L_1|)]$ is a condition which also appears when studying Lévy-driven CARMA processes, for which it is equivalent to the existence of a stationary solution~\cite{Brockwell2009existence,berger2019levydriven2}. Note that this condition is in particular implied by the finiteness of the absolute moment $\E[|L_1|^p]<\infty$ for some $p>0$, which is an equivalent characterization of tempered (\emph{i.e.}, bounded by a polynomial) Lévy processes~\cite{Fageot2014,Dalang2015Levy}.
\end{remark}

It is possible to construct Lévy processes whose entropy is infinite, even when the marginals $L_t$ are absolutely continuous for $t >0$. 
This is stated in the following proposition, whose proof is provided  in Appendix \ref{app:existenceentropy}. 

\begin{proposition}\label{prop:infiniteentropy}
There exists a Lévy process $L$ whose marginals $L_t$ are absolutely continuous for any $t>0$ and such that $\ent(\floor{L_1})= \infty$. In particular, the entropy of $L$ is infinite in the sense of Definition~\ref{def:finiteentropy}. 
\end{proposition}

Since all Lévy processes considered in this paper have marginals with finite differential entropy,
we will henceforth focus on characterizing the asymptotic behavior of $h(L_t)$ as $t\to 0^+$. 
By \eqref{eq:secondtrivial}, this informs the asymptotic behavior of $\H_{n,m}(L)$ (see the discussion after Corollary~\ref{prop:firstone}) and thus quantifies the compressibility of $L$. 
We start with the following observations, valid for all Lévy processes with finite differential entropy, and which will be refined in subsequent sections by considering the Blumenthal--Getoor index.

\begin{proposition}\label{prop:entropy-basic-properties}
	Let $L$ be a Lévy process such that $L_t$ has finite differential entropy for all $t>0$, then:
	\begin{enum}
	\item\label{it:entropy-increases} the function $t\mapsto h(L_t)$ is non-decreasing,
	\item\label{it:entropy-bound} $h(L_t) \leq h(L_1) + \frac{1}{2} \log \frac t {1-t}$ for $t\in (0,1)$.
	\end{enum}
\end{proposition}

\begin{proof}\begin{enum}
\item  {Let $X$ and $Y$ be independent random variables with finite differential entropy and $h(X|Y)$ be the conditional differential entropy (see~\cite[Section 8.4]{Cover2012elements}). 
Then, we recall that
\begin{equation} \label{eq:detailinequality}
    h(X) = h(X|Y) = h(X + Y | Y) \leq h(X+Y),
\end{equation}
 {where the second equality in \eqref{eq:detailinequality} is the conditional version of~\cite[Theorem 8.6.3]{Cover2012elements}}.} Since $L$ has independent increments we obtain, applying \eqref{eq:detailinequality} with $(X,Y) = (L_{t_1}, L_{t_2}-L_{t_1})$ for $0<t_1\leq t_2$, that $h(L_{t_1}) \leq h(L_{t_1} + L_{t_2} - L_{t_1})=h(L_{t_2})$.
\item For mutually independent variables $(X_1,\dots,X_n)$ with finite differential entropy, we have  by the entropy-power inequality $2h(\sum_{k=1}^n X_k)\geq \log\prn*{\sum_{k=1}^n e^{2h(X_k)}}$ (see \emph{e.g.}, \cite[Eq. (2)]{Verdu2006simple}). Applying this inequality to \eqref{eq:levy-division} with $X_k = L_{k/n} - L_{(k-1)/n}$, for which $h(X_k) = h(L_{1/n})$, we obtain for all $n\geq 1$ that
	\begin{equation}\label{eq:bound-integers}
		2 h(L_1)\geq \log n + 2h(L_{1/n})\,.
	\end{equation}
Consider $t\in(0,1)$, and define $n=\floor{1/t}$ so that $\frac{1}{n+1}< t\leq \frac{1}{n}$. Then,
\begin{displaymath}
	h(L_t) \leq h(L_{1/n}) \leq h(L_1) - \frac 1 2\log n  \leq h(L_1) + \frac 1 2\log\frac t{1-t}\,,
\end{displaymath}
where the first inequality follows from \ref{it:entropy-increases}, the second from \eqref{eq:bound-integers}, and the last is due to $n > (1-t)/t$ by definition of $n$.\qedhere
\end{enum}
\end{proof}

\begin{remark}
The asymptotic behavior of the upper bound in \ref{it:entropy-bound} is $h(L_1) + \frac{1}{2} \log t +o_t(1)$ as $t\to 0^+$. Since $h(W_t) = h(W_1) + \frac{1}{2}\log t$ for a Wiener process $W$, the bound in \ref{it:entropy-bound} is asymptotically tight for general Lévy processes. In Section \ref{sec:betanonzero}, we will obtain the tighter bound $h(L_t)\leq \frac{1}{\beta} \log t +O_t(1)$, for Lévy processes with Blumenthal--Getoor index $\beta$.
\end{remark}

\section{Entropy of Lévy Processes at Small Times}
\label{sec:betanonzero}

As per the discussion below \cref{prop:existenceentropy}, the entropic compressibility of a Lévy process $L$ with finite differential entropy is governed by the small-time asymptotics of $h(L_t)$, that we quantify in this section in terms of the Blumenthal--Getoor index $\beta$. We first consider the case of locally symmetric and self-similar Lévy processes, defined in \cref{sec:selfsim}, for which we show in \cref{sec:selfsimtheo} that $h(L_t)\sim\frac 1\beta \log t$ when $t \to 0^+$. We then consider general Lévy processes in Section~\ref{sec:upperboundentropy}, for which we obtain an upper bound on $h(L_t)$. As an important consequence, we deduce that when $\beta=0$, $h(L_t)$ diverges super-logarithmically when $t\to 0^+$.

\subsection{Locally Self-Similar Lévy Processes}
\label{sec:selfsim}

 \begin{definition} \label{def:locallystable}
  We say that a random process $X = (X_t)_{t \geq 0}$ is 
    \begin{itemize}[label=\raisebox{0.30ex}{\tiny$\bullet$}]
        \item {\emph{symmetric} if the $X$ and $-X = (-X_t)_{t \geq 0}$ have the same probability law;}
        \item \emph{self-similar of order $H \in \R$} if the rescaled random process $a^H X_{ \cdot / a} = ( a^H X_{ t / a} )_{t \geq 0}$ and $X$ have the same probability law for all $a > 0$;
        \item \emph{locally self-similar (LS) of order $H \in \R$} if 
       $a^H X_{ \cdot / a}$ converges in law to a non-trivial random process $Y$ when $a \rightarrow \infty$. If moreover the limiting process $Y$ is symmetric, we say that $X$ is \emph{locally symmetric and self-similar (LSS)}. 
 \end{itemize}
    
  \end{definition}

    Self-similar processes have been studied extensively~\cite{Sinai1976self,Embrechts2000introduction} and include the fractional Brownian motion~\cite{Mandelbrot1968} and generalizations of S$\alpha$S processes~\cite[Chapter 7]{Taqqu1994stable}. 
    Locally self-similar processes were introduced in~\cite[Definition 4.3]{fageot2019scaling}. 
    Higher values of $a>0$ in $a^H X_{\cdot / a}$ corresponds to zooming in the process at the origin. 
    The limiting process $Y$ in Definition~\ref{def:locallystable} is known to be self-similar of order $H$~\cite[Proposition 4.4]{fageot2019scaling}, which explains the terminology.
    
	The following result, proved in \cref{app:deferred-betanonzero}, characterizes symmetric self-similar and locally symmetric and self-similar random processes among the family of Lévy processes.

\begin{proposition} \label{prop:stableandlocstable}
  A Lévy process is symmetric and self-similar if and only if it is a S$\alpha$S process. In this case, the self-similarity order is $H = 1/\alpha$.
  A Lévy process with Blumenthal--Getoor index $\beta$ is LSS if and only if $\beta > 0$ and its characteristic  exponent $\Psi$ satisfies
  \begin{equation} \label{eq:psilimitforcarac}
	  \Psi(\xi) \underset{\abs \xi \to \infty}{\sim} - \gamma \lvert \xi \rvert^\beta
  \end{equation}
  for some constant $\gamma > 0$. The local self-similarity order is then $H = 1 / \beta$ and the limiting process $Y$ is S$\alpha$S with $\alpha = \beta$. 
  \end{proposition}

\begin{remark}
A locally symmetric self-similar process is not necessarily symmetric itself, but its potential skewness vanishes at small times. Proposition \ref{prop:stableandlocstable} can be extended  to the non-symmetric case, but we will not cover it in this paper (see Section \ref{sec:discuss} for a short discussion regarding this point). 
\end{remark}

\begin{remark}
According to Proposition~\ref{prop:stableandlocstable}, the Blumenthal--Getoor index $\beta$ of a LSS Lévy process is necessarily positive. In fact, a Lévy process $L$ with $\beta = 0$ is such that $a^H X_{ \cdot / a}$ converges in law to $0$ for all $H \in \R$ when $a \to \infty$~\cite[Proposition 4.7]{fageot2019scaling}.
\end{remark}
     
\subsection{A Small-time Limit Theorem for the Entropy of LSS Lévy Processes}
\label{sec:selfsimtheo}

In this section, we consider the differential entropy $h(L_t)$ of an LSS Lévy process $L$, for which we obtain an exact asymptotic expansion as $t\to 0^+$ up to a vanishing term.

\begin{theorem} \label{theo:maintheolocallyselfsim}
Let $(L_t)_{t\geq 0}$ be a LSS Lévy process with Blumenthal--Getoor index $\beta\in(0,2]$ and such that $L_1$ has a finite absolute moment $\E[|L_1|^p]< \infty$ for some $p>0$.  Then, the differential entropy $h(L_t)$ is finite for all $t > 0$ and
\begin{equation}\label{eq:hL1n}
    h(L_{t} ) = \frac 1\beta \log t +  h(Y_\beta) + o_{t}(1), 
\end{equation}
where $Y_{\beta}$ is a S$\alpha$S random variable with $\alpha = \beta$ and $o_t(1) \to 0$ when $t\to 0^+$.
\end{theorem}

We shall use the following lemma in the proof of Theorem~\ref{theo:maintheolocallyselfsim}.

\begin{lemma} \label{lemma:controlmoments}
Let $L$ be a Lévy process with characteristic exponent $\Psi$, Blumenthal--Getoor index $\beta > 0$, and satisfying $\E [ |L_1|^p ] < \infty$ {for some $p > 0$}. We assume moreover that $\Psi$ is asymptotically dominated by $|\cdot |^\beta$, \textit{i.e.}, there exists a constant $C > 0$ such that $\lvert \Psi(\xi) \rvert \leq C |\xi|^\beta$ for all $|\xi|\geq 1$.
Then, for all $0< q<\min\set{1, p,\beta}$, there exists a constant $M_q >0$ such that 
    \begin{equation}
        \label{eq:boundMpforLt}
        \E [|t^{- 1 / \beta}L_t|^q ] \leq M_q, \qquad \forall 0 < t \leq 1.
    \end{equation}
\end{lemma}

Upper bounds such as \eqref{eq:boundMpforLt} are known as moment estimates and play a crucial role when studying the local smoothness of the sample paths of Lévy processes~\cite{aziznejad2018wavelet,Fageot2017multidimensional,Bottcher2014levy}. Moment estimates and extensions thereof have been studied by several authors~\cite{Luschgy2008moment,Deng2015shift,Kuhn2017existence}.
Lemma \ref{lemma:controlmoments} can be deduced from known estimates. In particular, it is a consequence of~\cite[Theorem 2.9]{kuhn2017levy}, which considers the more general class of Lévy-type processes. For completeness, a self-contained proof in the case of Lévy processes is provided in Appendix \ref{app:deferred-betanonzero}.
Note that the hypothesis that $\Psi$ is asymptotically dominated by $\lvert \cdot \rvert^{\beta}$ is automatically satisfied for LSS processes (by Proposition \ref{prop:stableandlocstable}) but is not true for arbitrary Lévy processes. 

\begin{proof}[Proof of Theorem~\ref{theo:maintheolocallyselfsim}]
Fix a sequence $(t_n)_{n\geq 1}$ of positive reals such that $\lim_{n\to\infty} t_n=0$ and define $X_n = t_n^{-1/\beta}L_{t_n}$ for $n\geq 1$.  We will prove that $\lim_{n\to\infty} h(X_n) = h(Y_\beta)$ which implies that $h(t^{-1/\beta} L_t)\to h(Y_\beta)$ as $t\to 0^+$ since the sequence $(t_n)_{n\geq 1}$ is arbitrary. Equation~\eqref{eq:hL1n} then follows since $h(aX) = \log a + h(X)$ for $a > 0$. 

    Due to the sector condition \eqref{eq:sector}, we have that
    \begin{displaymath}
		c\cdot|\Psi(\xi)|  = c\cdot\sqrt{ \prn*{\Re \Psi(\xi)} ^2 + \prn*{\Im \Psi(\xi)} ^2 }
		\leq |\Re \Psi(\xi)| = - \Re \Psi(\xi),
    \end{displaymath}
	where we defined $c = (1+C^2)^{-1/2}$. Hence for all $n\geq 1$ and $\xi\in\R$,
	\begin{align}\label{eq:first}
	    |	\Phi_{X_n}(\xi) | 
	    &= | \Phi_{L_{t_n}}(t_n^{-1/\beta}\xi) |
    	= |\exp\of[\big]{t_n\Psi(t_n^{-1/\beta}\xi)} |
        = \exp\of[\big]{t_n \Re \Psi(t_n^{-1/\beta}\xi)}  \nonumber \\
        &\leq \exp\of[\big]{- c \cdot t_n |\Psi(t_n^{-1/\beta}\xi)|},
	\end{align}
	where the third equality uses that $|\exp z| = \exp( \Re z)$ for $z\in\mathbb{C}$.

By \cref{prop:stableandlocstable}, the characteristic exponent $\Psi$ of $L=(L_t)_{t\geq 0}$ satisfies $|\Psi(\xi)|\sim\gamma\abs\xi^\beta$ as $\abs\xi\to\infty$. In particular, there exists $B\geq 0$ such that $|\Psi(\xi)| \geq \gamma\abs\xi^\beta/2$ for $\abs\xi>B$. Let us now define the function $g:\R\to\R$ by $g(\xi) = 1$ for $\abs \xi \leq B$ and $g(\xi) = e^{-c\gamma\abs\xi^\beta/2}$ for $\abs \xi> B$. For $n$ large enough such that $t_n\leq 1$ we have, for all $\xi\in\R$,
\begin{equation} \label{eq:smallthingtoprove}
	\exp\of[\big]{-c \cdot  t_n |\Psi(t_n^{-1/\beta}\xi)|} \leq g(\xi).
\end{equation}
Indeed, \eqref{eq:smallthingtoprove} is obvious for  $\abs \xi\leq B$ since $\exp(-x)\leq 1$ for $x\geq 0$.  For $\abs\xi > B$, we also have $\abs{t_n^{- 1/\beta} \xi} > B$ (since $0 < t_n \leq 1$), hence $c \cdot t_n |\Psi(t_n^{-1/\beta}\xi)| \geq  c \cdot  t_n \gamma \abs{ t_n^{- 1/\beta} \xi }^{\beta} / 2 = c \cdot  \gamma \abs\xi ^\beta / 2 $ and \eqref{eq:smallthingtoprove} readily follows by definition of $g$.

By assumption on $L$, $X_n$ converges in distribution to $Y_\beta$ which implies the pointwise convergence $\Phi_{X_n}(\xi)\to\Phi_{Y_\beta}(\xi)$ for all $\xi\in\R$ as $n\to\infty$. Equation~\eqref{eq:first} together with \eqref{eq:smallthingtoprove} implies that $\Phi_{X_n}$ is uniformly dominated by the integrable function $g$ for $n$ large enough. Therefore, $\Phi_{X_n}\to\Phi_{Y_\beta}$ in $L_1$-norm by Lebesgue's dominated convergence theorem.

Denote by $q_\beta$ the pdf of $Y_\beta$ and by $p_n$ the pdf of $X_n$ for $n\geq 1$ (those are well-defined since $\Phi_{X_n}$ and $\Phi_{Y_\beta}$ are integrable). Writing $p_n$ and $q_\beta$ as the inverse Fourier transforms of $\Phi_{X_n}$ and $\Phi_{Y_\beta}$ and using the triangle inequality, we obtain for all $x\in\R$ and $n\geq 1$,
	\begin{displaymath}
	2\pi\abs[\big]{p_n (x)  -q_\beta(x)}
	=\abs*{\int_\R  \left( \Phi_{X_n}(\xi)e^{-ix\xi} - \Phi_{Y_\beta}(\xi)e^{-ix\xi} \right) \rd\xi}
    \leq\int_\R \abs[\big]{\Phi_{X_n}(\xi) - \Phi_{Y_\beta}(\xi)}\rd\xi\,.
	\end{displaymath}
Since $\Phi_{X_n}\to\Phi_{Y_\beta}$ in $L_1$-norm, this implies the pointwise---in fact, uniform---convergence $p_n(x)\to p(x)$ for all $x\in\R$ as $n\to\infty$ (this fact is well-known, see for instance~\cite[Corollary 1.2.4]{ushakov2011selected}). By {Scheffé's lemma~\cite[Section 5.10]{williams1991probability}}, this in turns implies the convergence $p_n\to p$ in $L_1$-norm.

Writing $p_n$ as the inverse Fourier transform of $\Phi_{X_n}$ as above, we get $2\pi\norm{p_n}_{\infty}\leq \int_\R \abs{\Phi_{X_n}}\leq \int_\R \abs g$ for $n$ large enough and similarly for $q_\beta$. Finally, according to Lemma~\ref{lemma:controlmoments}, for  $0<q<\min\set{1,\beta,p}$, there exists $M_q>0$ such that $\E[\abs{X_n}^q]\leq M_q$ for $n$ large enough such that $t_n \leq 1$. We then conclude using \cite[Theorem 1]{Ghourchian2017existence} that $\lim_{n\to\infty} h(X_n) = h(Y_\beta)$.
\end{proof}

\begin{example}
	If $L^\alpha$ is an S$\alpha$S process, it is obviously LSS and \cref{theo:maintheolocallyselfsim} then implies that
\begin{displaymath}
    h(L_t^\alpha) = \frac{1}{\alpha} \log t + h(L_1^\alpha) + o_t(1),
\end{displaymath}
where $o_t(1)\to 0$ when $t\to 0^+$. This specific case can be proved directly using the stability property as was done by Ghourchian et al. in~\cite[Proposition 1]{Ghourchian2017compressible}. \cref{theo:maintheolocallyselfsim} thus significantly generalizes this result to all LSS processes.
\end{example}

\begin{remark}[Link with local limit theorems]
	The key step in the proof of \cref{theo:maintheolocallyselfsim} is to prove the convergence of the sequence $(X_n)_{n\geq 1}$ to $Y_\beta$ not only in law but in the (stronger) sense that the sequence of probability density functions $(p_n)_{n\geq 1}$ converges uniformly to the density $p$ of $Y_\beta$. This is closely related to the literature on \emph{local limit theorems}, whose goal is to characterize the convergence of normalized sums of i.i.d.\ random variables in terms of the convergence of the pdfs, and includes notably uniform convergence~\cite{gnedenko1954limit}, convergence in total variation~\cite{prokhorov1952local}, and convergence in relative entropy, {which we further discuss in the next remark.}
    See also~\cite{bobkov2019local} for a recent unifying perspective on local limit theorems.
\end{remark}

\begin{remark}[A local entropic CLT]
	{The study of entropic central limit theorems (CLT) has been initiated by Andrew Barron for the Gaussian case~\cite{barron1986entropy} and further generalized to non-Gaussian stable limits by Sergey Bobkov, Gennadiy Chistyakov, and Friedrich Götze~\cite{bobkov2011convergence}. For instance, Barron's contribution has been to show that for a sequence $(X_n)_{n\geq 1}$ of finite-variance i.i.d.\ random variables with finite differential entropy, the normalized sum $S_n=\sum_{k=1}^n X_k/\sqrt{n}$  is such that $h(S_n) \rightarrow_{n \rightarrow \infty} h(G)$ with $G \sim \mathcal{N}(0,\sigma^2)$ where $\sigma^2$ the common variance of the $X_n$.}

   {Let $L$ be a Lévy process and let $X_k = L_k - L_{k-1}$ be the increments at integral times for $k \geq 1$. Then, we have that $L_n = \sum_{k=1}^n X_k$ is a sum of i.i.d.\ random variables and entropic CLTs readily apply to the study of the asymptotic behavior of $L_t$ for $t \rightarrow \infty$. 
   It is worth noting that sequences of i.i.d. random variables studied in the context of entropic CLTs (or “standard” CLTs~\cite[Chapter VIII]{Feller2008introduction}) are usually assumed to be in the basin of attraction of a stable law~\cite{barron1986entropy,bobkov2011convergence}. In the context of Lévy processes, this is equivalent to the fact that the Lévy process $L$ is asymptotically self-similar, as defined in~\cite[Definition 4.3]{fageot2019scaling}, or equivalently that the characteristic exponent $\Psi$ of $L$ satisfies $\Psi(\xi) \sim_{|\xi|\rightarrow 0} - \gamma |\xi|^\alpha$ for some $0 < \alpha \leq 2$ and $\gamma > 0$.} 
   
   {In contrast, Theorem~\ref{theo:maintheolocallyselfsim} characterizes the asymptotic behavior of the differential entropy of $L_t$ at small times $t \rightarrow 0^+$ and our work can thus be seen as the local counterpart of classic entropic central limit theorems. The assumption of \emph{local} self-similarity, equivalent to $\Psi(\xi) \sim_{|\xi|\rightarrow \infty}-\gamma |\xi|^\alpha$, thus plays the same role for characterizing the asymptotic behavior of $L_t$ for small $t\to 0^+$ as the self-similarity assumption for $t\to\infty$.
   Further discussion about the local versus asymptotic properties of Lévy processes and their generalizations can be found in~\cite[Chapter 7]{FageotThese}.}
\end{remark}

\subsection{An Upper Bound on the Small-Time Entropy of Lévy Processes} 
\label{sec:upperboundentropy}

The specific case of LSS Lévy processes studied in the previous section excludes both Lévy processes with $\beta > 0$ which are not LSS and Lévy processes with $\beta=0$. In this section, we consider a general Lévy process $L$ and upper bound the decay of $h(L_t)$ at small times. Specifically, the following theorem gives an upper bound on the limit superior of $h(L_t) / \log(1/t)$ when $t\to 0^+$ in terms of the Blumenthal--Getoor index $\beta \in[0, 2]$. This is achieved by ``approximating'' $L$ with an LSS process and applying \cref{theo:maintheolocallyselfsim}.

\begin{theorem} \label{theo:submodularstuff} 
	Let $L = (L_t)_{t\geq 0}$ be a Lévy process with Blumenthal--Getoor index $\beta \in [0,2]$ and such that $L_1$ has a finite absolute moment $\E[\abs{L_1}^p]< \infty$ for some $p>0$. Then,
\begin{equation} \label{eq:betterone}
	\limsup_{t\to 0^+} \frac{h(L_{t})}{\log (1/t)} \leq - \frac{1}{\beta},
\end{equation}
with the convention $1/0 = \infty$. 
In particular, for $\beta = 0$, we have that
\begin{equation} \label{eq:theoneforbetazero}
    \underset{t \rightarrow 0^+}{\lim} \frac{ h(L_{t}) }{\log (1/t)}  = - \infty.
\end{equation}
\end{theorem}

\begin{proof}
	The case $\beta=2$ is covered by \cref{prop:entropy-basic-properties}~\ref{it:entropy-bound} so we henceforth assume that $\beta\in[0,2)$. For $\alpha\in ( \beta,2]$, let $L^\alpha$ be a S$\alpha$S process with parameter $\alpha$, independent of $L$, and define $\widetilde L = L + L^\alpha$. We denote by $\widetilde\Psi$, $\Psi$, and $\Psi^\alpha$ the characteristic exponents of $\widetilde L$, $L$, and $L^\alpha$ respectively.  By definition, there exists $\gamma > 0$ such that $\Psi^\alpha(\xi) = -  \gamma \abs\xi^\alpha$. Moreover, since $\beta < \alpha$ and by definition of $\beta$, we have that $\abs{\Psi (\xi)}/ \abs\xi^\alpha\to 0$ when $\abs\xi\to\infty$. This implies that $\widetilde\Psi(\xi)\sim - \gamma\abs\xi^\alpha$ when $\abs\xi\to\infty$, hence $\widetilde L$ is LSS by Proposition~\ref{prop:stableandlocstable}.

For two independent absolutely continuous random variables $X,Y$ such that $(X+Y)$ has finite differential entropy, we have that\footnote{The finiteness of $h(X+Y)$ implies that $h(X)$ is well-defined in $[-\infty,\infty)$. This is therefore the case for $X=L_t$ in the proof. Note that the theorem is trivially true when $h(L_t) = -\infty$.} $h(X) = h(X+Y|Y) \leq h(X+Y)$. Hence
\begin{equation}\label{eq:entropy-monotone}
	h(L_t) \leq h(L_t + L^\alpha_t) = h(\widetilde L_t) =  - \frac 1\alpha\log(1/t) + O_t(1),
\end{equation}
where the last equality is by \cref{theo:maintheolocallyselfsim} applied to $\widetilde L$ (the assumptions of \cref{theo:maintheolocallyselfsim} are easily satisfied from the assumptions on $L$).  Dividing by $\log(1/t)$ for $t<1$ and taking the limit superior on both sides of \eqref{eq:entropy-monotone} yields
\begin{displaymath}
	\limsup_{t\to 0^+} \frac{ h(L_{t})}{\log(1/t)}
	\leq \limsup_{t\to 0^+} \left\{ -\frac 1\alpha + o_t(1)\right\}  = - \frac 1\alpha,
\end{displaymath}
which concludes the proof after letting $\alpha\to\beta$.
\end{proof}

\section{Examples and Applications}
\label{sec:examples}

We apply our previous findings to specific subfamilies of Lévy processes. We consider non-stable Lévy processes with positive Blumenthal-Getoor indices in Section~\ref{sec:layeredtempered} and gamma processes, for which $\beta=0$, in Section~\ref{sec:gammalaplace}. 

   \subsection{Layered and Tempered Stable Processes} \label{sec:layeredtempered}
    
   Layered stable processes and tempered stable processes were respectively introduced by Christian Houdré and Reiichiro Kawai~\cite{Houdre2007layered} and by Jan Rosinski~\cite{Rosinski2007tempering}. In both cases, we characterize the asymptotic behavior of the entropy, thus demonstrating the applicability of Theorem~\ref{theo:maintheolocallyselfsim} to Lévy processes that were not covered by previous results. For simplicity, we restrict ourselves to symmetric Lévy processes.
   
\subsubsection*{Layered Stable Processes} 

\begin{definition} \label{def:layered}
A symmetric Lévy process $L$ is said to be a \emph{layered stable process} with indices $(\alpha_0, \alpha_\infty) \in (0,2) \times (0, \infty)$
if it has no Gaussian part ($\sigma^2 = 0$ in \eqref{eq:LevyKhintchine}) and if its Lévy measure $\nu$ can be written $ \nu(\mathrm{d} t) = q(\abs t) \mathrm{d}t$ {for a continuous function $q : \R^+\setminus\set 0 \to \R^+$ with}
\begin{equation} \label{eq:qlayered}
    q(t) 
\underset{t \rightarrow 0}{\sim}    
  \frac{a_0}{t^{\alpha_0 + 1}} \quad \text{and} \quad   
    q(t) 
\underset{t \rightarrow \infty}{\sim}    
  \frac{a_\infty}{t^{\alpha_\infty + 1}}
\end{equation}
for come constants $a_0, a_\infty > 0$. 
\end{definition}

Definition~\ref{def:layered} is adapted from \cite[Definition 2.1]{Houdre2007layered} to the ambient dimension $d = 1$. Stable processes correspond to the specific case where $q(t) = \frac c {\abs t^{\alpha+1}}$ for some constant $c>0$ and $\alpha\in (0,2)$. As was studied in details in~\cite{Houdre2007layered}, the index $\alpha_0$ governs the short-term behavior of the Lévy process, while $\alpha_\infty$ captures long-terms structures.
One should therefore not be surprised to see in the next proposition that the entropy of a layered stable process is asymptotically determined by its index $\alpha_0$. 

\begin{proposition}
\label{prop:layered}
Let $L$ be a layered stable process with parameters $(\alpha_0, \alpha_\infty) \in (0,2) \times (0, \infty)$.
Then, $L$ is LSS with Blumenthal--Getoor index $\beta=\alpha_0$ and
\begin{equation} \label{eq:Hnmlayered}
    h(L_t) = \frac{1}{\alpha_0} \log t + h + o_t(1)
\end{equation}
for some constant $h \in \R$, where $o_t(1)$ vanishes for $t \rightarrow 0^+$.
\end{proposition}
 
 \begin{proof}
  From \eqref{eq:LevyKhintchine}, we easily deduce that the characteristic exponent of a symmetric Lévy process without Gaussian part is given by 
    \begin{equation}\label{eq:caractexposimple}
    \Psi(\xi) = \int_{\R}  (\cos ( t \xi ) - 1 )  \nu( \mathrm{d} t ), \quad \forall \xi \in \R.
    \end{equation}
    We decompose $\Psi = \Psi_0 + \Psi_\infty$  with
\begin{displaymath}
    \Psi_0(\xi) = \int_{|t|\le 1}  (\cos ( t \xi ) - 1 )  \nu( \mathrm{d} t )
	\quad\text{and}\quad
    \Psi_\infty(\xi) = \int_{|t|> 1}  (\cos ( t \xi ) - 1 )  \nu( \mathrm{d} t ).
\end{displaymath}
 Then, after the change of variable $u=t\xi$ for $\xi\neq 0$, we get
    \begin{displaymath}
        \Psi_0(\xi) = - \left( \int_{\R} (1 -\cos u) \frac{q(|u| / \xi)}{\lvert \xi \rvert^{\alpha_0+1}}  \One_{\lvert u \rvert \leq \lvert \xi\rvert} \mathrm{d}u \right) \lvert \xi\rvert^{\alpha_0}
        := - \left( \int_{\R} g_\xi (u) \mathrm{d} u \right) \lvert \xi\rvert^{\alpha_0} . 
    \end{displaymath}
	Then, due to \eqref{eq:qlayered} for $t\rightarrow 0^+$, we have that $g_\xi(u) \rightarrow (1-\cos u) / |u|^{\alpha_0+1}$ for all $u \in \R$ when $\abs\xi\to \infty$. Moreover, again according to \eqref{eq:qlayered} and using the continuity of $q$ over $\mathbb{R}\setminus\{0\}$, there exists $M>0$ such that $q(t) \leq M / |t|^{\alpha_0+1}$ for all $\abs t \leq 1$. Hence, we deduce that
	\begin{displaymath}
        0\leq g_\xi(u)  \leq \One_{|u|\leq |\xi|}  \frac{M (1-\cos u)}{|u|^{\alpha_0+1}}  
        \leq \frac{M (1-\cos u)}{|u|^{\alpha_0+1}},
    \end{displaymath}
    the latter being integrable on $\R$. Lebesgue's dominated convergence theorem then ensures that $\int_{\R} g_\xi(u) \mathrm{d} u \rightarrow \gamma := \int_{\R} \frac{ (1-\cos u)}{|u|^{\alpha_0+1}} \mathrm{d}u$ when $|\xi|\rightarrow \infty$, and finally, 
    \begin{equation} \label{eq:psi0limit}
        \Psi_0(\xi) \underset{|\xi|\rightarrow \infty}{\sim} - \gamma \lvert \xi\rvert^{\alpha_0}.
    \end{equation}
    Moreover, \eqref{eq:qlayered} and the continuity of $q$ implies the existence of $M>0$ such that $q(|t|) \leq M / |t|^{\alpha_\infty + 1}$ for all $|t|>1$. This implies that
    \begin{equation}\label{eq:psiinftylimit}
		\lvert \Psi_\infty(\xi) \rvert =  \int_{|t|> 1}  \big(1 - \cos (\xi t)\big) q(t) \mathrm{d} t 
         \leq 2 M \int_{|t|> 1}  \frac{\mathrm{d} t}{|t|^{\alpha_\infty + 1}},
    \end{equation}
    which is a finite constant independent of $\xi \in \R$. 

    Combining \eqref{eq:psi0limit} and \eqref{eq:psiinftylimit}, we deduce that 
        $\Psi(\xi)\sim- \gamma \lvert \xi\rvert^{\alpha_0}$ as $\abs\xi\to\infty$
	which implies by  \cref{prop:stableandlocstable} that $L$ is LSS with Blumenthal--Getoor index $\beta = \alpha_0$.

	Finally, $\E[|L_1|^p]<\infty$ for all $p<\alpha_\infty$ by \cite[Proposition 2.1]{Houdre2007layered}. Hence $L$ satisfies all the assumptions of \cref{theo:maintheolocallyselfsim} and we conclude that \eqref{eq:Hnmlayered} holds.
    \end{proof}   

\subsubsection*{Proper Tempered Stable Processes} 

\begin{definition}
A symmetric Lévy process is said to be a \emph{proper tempered stable process} with parameters $p > 0$ and $0 < \alpha < 2$ if it has no Gaussian part and if its Lévy measure $\nu$ can be written as
\begin{equation} \label{eq:nutempered}
    \nu ( \mathrm{d} t ) = 
    q(\lvert t \rvert^p) 
    \frac{\mathrm{d} t}{\lvert t \rvert^{\alpha+1}} 
\end{equation}
for some completely monotone function\footnote{See \cite{schilling2012bernstein} for a general reference on completely monotone functions.} $q: (0,\infty) \rightarrow (0,\infty)$ such that
\begin{equation} \label{eq:tructruc}
    \int_0^1 x^{1-\alpha} q(x) \mathrm{d} x < \infty, \quad
    \int_1^\infty x^{-1-\alpha} q(x) \mathrm{d} x < \infty,\quad
    q(x) \underset{x \rightarrow \infty}{\longrightarrow} 0,  \
    \text{ and }   \ 
    q(x) \underset{x \rightarrow 0^+}{\longrightarrow} c  
\end{equation}
{for some constant $0 < c < \infty$.}
\end{definition}

For this definition, we follow the monograph~\cite{grabchak2016tempered}, in particular Definitions 3.1 and 3.2 specialized to the ambient dimension $d=1$. Stable processes correspond to the case where $q$ is identically equal to a positive constant, although this is excluded by the condition that $q$ vanishes at infinity. The effect of $q$ is to temper the asymptotic behavior of $\nu$, which impacts the heavy-tailedness of the Lévy process. We refer the interested reader to \cite[Chapter 1]{grabchak2016tempered} for practical motivations and for historical references on tempered stable processes. 

\begin{proposition}\label{prop:temperedstable}
Let $L$ be a proper tempered stable process with parameters $p > 0$ and $0 < \alpha < 2$. Then, $L$ is LSS with Blumenthal--Getoor index $\beta=\alpha$ and
\begin{equation}\label{eq:entropy-layered}
    h(L_t) = \frac{1}{\alpha} \log t + h + o_t(1)
\end{equation}
for some constant $h \in \R$, where $o_t(1)$ vanishes for $t \rightarrow 0^+$.
\end{proposition}

{The proof of Proposition \ref{prop:temperedstable} is very similar to the one of Proposition \ref{prop:layered} and is given in Appendix~\ref{app:deferred-examples}.}

\subsection{Gamma Processes} \label{sec:gammalaplace}

In the case where $\beta=0$, \cref{theo:submodularstuff} guarantees that $h(L_t)$ diverges to $-\infty$ super-logarithmically. The decay rate can in fact be significantly faster as illustrated in this section by the gamma process, for which we show that it is asymptotically equivalent to $-\frac{1}{\tau t}$ for some constant $\tau>0$.

The gamma distribution with scale parameter $\theta>0$ and shape parameter $\tau>0$ has characteristic function given by $\xi\mapsto 1/(1-i\theta\xi)^\tau$ and is thus infinitely divisible. It defines a Lévy process $L=(L_t)_{t\geq 0}$, called a \emph{gamma process}, whose characteristic exponent satisfies $|\Psi(\xi)| = |\tau \log(1-i\theta\xi)|\sim \tau\log|\xi|$ as $\abs\xi\to\infty$. Hence its Blumenthal--Getoor index is zero. 
	
\begin{proposition} \label{prop:gamma}
Let $L = (L_t)_{t\geq 0}$ be a gamma process with scale parameter $\theta > 0$ and shape parameter $\tau > 0$. Then we have as $t\to 0^+$,
	\begin{equation} \label{eq:hLtgamma}
		h(L_t) = -\frac 1{\tau t} -\log t + \log\frac\theta\tau -\gamma + 1 + o_t(1).
	\end{equation}
\end{proposition}

    \begin{proof}
    For all $t>0$, $L_t$ is gamma-distributed with scale parameter $\theta$ and shape parameter $\tau t$, hence its differential entropy is given by (see \emph{e.g.}, \cite[Table 7.2]{Crooks2019field}):
	\begin{equation}\label{eq:entropy-gamma}
		h(L_t) = \tau t + \log \theta + \log \Gamma(\tau t) + (1-\tau t) \psi(\tau t),
	\end{equation}
	where $\Gamma$ and $\psi$ are the gamma and digamma functions respectively. Recall that for all $x>0$, $x\Gamma(x) = \Gamma(x+1)$ and $\psi(x) + 1/x = \psi(x+1)$. Hence as $t\to 0^+$,
	\begin{displaymath}
		\begin{split}
			\log \Gamma(\tau t) &= \log\frac{\Gamma(1+\tau t)}{\tau t} \\
								  &= -\log(\tau t) + o_t(1)
	\end{split}
	\quad\quad\text{and}\quad\quad
	\begin{split}
		\psi(\tau t) &= \psi(\tau t + 1) - \frac 1{\tau t}\\
					   &=-\frac 1{\tau t} -\gamma + o_t(1)
	\end{split}\,,
	\end{displaymath}
	where we used that $\Gamma$ and $\psi$ are continuous at $1$ with $\Gamma(1) = 1$ and $\psi(1) = -\gamma$, for $\gamma$ the Euler--Mascheroni constant. We thus obtain from \eqref{eq:entropy-gamma} the asymptotic expansion \eqref{eq:hLtgamma}.
	\end{proof}

\section{Discussion and Conclusion}
\label{sec:discuss}

\subsection{The Entropic Compressibility Hierarchy of Lévy processes}
\label{sec:hierarchy}

In this paper, we studied the entropy $\mathcal{H}_{n,m}(L)$ of a Lévy process $L$, as introduced in~\cite{Ghourchian2017compressible}. We restricted our study to Lévy processes whose marginals are absolutely continuous with finite differential entropy. In this case, we used the fact that (see Corollary~\ref{prop:firstone})  
\begin{displaymath}
    \frac{\mathcal{H}_{n,m}(L)}{n} - \log m = h(L_{1/n}) + \epsilon_{n,m}
\end{displaymath}
with $\epsilon_{n,m} \rightarrow 0$ when $m\rightarrow \infty$ for fixed $n \geq 1$ 
to focus our attention to the small-time behavior of the differential entropy $h(L_t)$, which satisfies $h(L_t) \to - \infty$ when $t \rightarrow 0^+$. The behavior of $h(L_t)$ has an interpretation in terms of compressibility: the faster it diverges, the more compressible $L$ is. 
Thus, Theorems \ref{theo:maintheolocallyselfsim} and \ref{theo:submodularstuff} reveal an entropy-based compressibility hierarchy for Lévy processes with absolutely continuous marginals, determined by the Blumenthal--Getoor index, therefore generalizing~\cite[Thm. 4]{Ghourchian2017compressible} in a substantial manner. This hierarchy is summarized in Figure~\ref{fig:hierarchy}.

\begin{figure}[t] \centering
\usetikzlibrary{arrows.meta}
	\begin{tikzpicture}[>=Stealth]
		\node[below, align=center, yshift=0.2em] at (1, 4)
		{Wiener process  \\[0.5em]
		$\displaystyle \frac 1 2\log t$};

		\node[below, align=center] at (1, 2.2)
		{Lévy process\\ with $\beta=2$\\[0.5em]
		$\displaystyle \frac 1 2\log t$};
		
		\node[below, align=center] at (5.7,4.5)
		{$S\alpha S$ process\\ with $0<\alpha< 2$\\[0.5em]
		$\displaystyle \frac 1\alpha \log t$};

		\node[below, align=center] at (9.7,4.5)
		{Lévy process\\with $\beta=0$\\[0.5em] $\ll \log t$};
		
		\node[below, align=center] at (5.7, 2.2)
		{Lévy process\\ with $0<\beta<2$\\[0.5em]
		$\displaystyle \frac 1 \beta \log t$};

		\node[above left, align=center,yshift=0.22em] at (12, 0)
			{Gamma process\\[0.5em]$\displaystyle -\frac 1 {\tau t}$};

		\draw[->] (0,0) -- node[below,align=center]{entropic compressibility}
		(12, 0) node[below left]{sparser};
\end{tikzpicture}
\caption{Entropic compressibility of Lévy processes with absolutely continuous marginals. The expressions shown are upper-bounds for general Lévy processes and equal to the dominant term of the asymptotics of $h(L_{t})$ for locally symmetric and self-similar processes. Wiener processes and S$\alpha$S stable processes were already covered in \cite{Ghourchian2017compressible}.}
\label{fig:hierarchy}
\end{figure}

\subsection{Future Directions}\label{sec:open}

We now discuss the assumptions made in the present manuscript and suggest open questions which may be investigated in future work.

\vspace{1em}\noindent\emph{Symmetry.} We only considered locally \emph{symmetric} and self-similar Lévy processes in Section~\ref{sec:betanonzero}. This choice was mostly made for convenience and significantly simplifies the exposition. For instance, symmetric stable processes are parameterized by two parameters, their characteristic exponent being of the form $\xi \mapsto  - \gamma \lvert \xi \rvert^\alpha$. In contrast, the characterization of all stable laws requires two additional parameters, including a skewness parameter, and is thus more involved (see \cite[Definition 1.1.6]{Taqqu1994stable}). 
It should be possible to generalize Theorem~\ref{theo:maintheolocallyselfsim} to locally self-similar processes with possibly non-symmetric local limit, but this requires extending the currently used results from \cite{fageot2019scaling} to non-symmetric limits.
Note however that the asymptotic entropy of any (possibly non-symmetric) stable random process was obtained in~\cite{Ghourchian2017compressible}, and that Theorem \ref{theo:submodularstuff} is valid for arbitrary Lévy processes.
    
\vspace{1em}\noindent\emph{Local Self-similarity.} For $\beta > 0$, we only quantified the small-time entropy of Lévy processes that are locally self-similar. It is however possible to consider Lévy processes that are not, and thus whose characteristic exponent violates \eqref{eq:psilimitforcarac} by Proposition~\ref{prop:stableandlocstable}. Examples of such processes have been constructed by Walter Farkas, Niels Jacob, and René Schilling in \cite[Examples 1.1.14 and 1.1.15]{Farkas2001function}. It would be interesting to quantify the small-time evolution of $h(L_t)$ for such Lévy processes.

    We remark however that Theorem~\ref{theo:submodularstuff} applies to Lévy processes that are not self-similar, and implies in particular that a Lévy process with absolutely continuous marginals and Blumenthal--Getoor index $\beta \in (0,2)$ is at least as compressible as S$\alpha$S processes with $\alpha \geq \beta$.
    
\vspace{1em}\noindent\emph{The case $\beta = 0$.} We obtained an exact asymptotic expansion for the entropy of gamma processes in Propositions \ref{prop:gamma}.
Beyond this specific example, our current hierarchy does not distinguish between Lévy processes with $\beta=0$ and simply states that the entropy $h(L_{t})$ diverges super-logarithmically when $t\to 0^+$. A possible research direction would be to refine the compressibility hierarchy for Lévy processes with $\beta = 0$.

\section*{Acknowledgments}
The authors are grateful to Shayan Aziznejad and Arash Amini for fruitful discussions in the early days of this project and to Iosif Pinelis for the construction of an infinitely divisible random variable with infinite differential entropy (cf.\ Proposition \ref{prop:infiniteentropy}).
Julien Fageot was supported by the Swiss National Science Foundation (SNSF) under Grant \texttt{P2ELP2\_181759}. Alireza was partially supported by the MathWorks Engineering Fellowship.

\bibliographystyle{plain}
\bibliography{references}

\begin{thebibliography}{10}

\bibitem{Amini2011compressibility}
A.~Amini, M.~Unser, and F.~Marvasti.
\newblock Compressibility of deterministic and random infinite sequences.
\newblock {\em IEEE Transactions on Signal Processing}, 59(11):5193--5201,
  2011.

\bibitem{Applebaum2009levy}
D.~Applebaum.
\newblock {\em L{\'e}vy {P}rocesses and {S}tochastic {C}alculus}.
\newblock Cambridge {U}niversity {P}ress, 2009.

\bibitem{aziznejad2018wavelet}
S.~Aziznejad and J.~Fageot.
\newblock Wavelet analysis of the {B}esov regularity of {L}{\'e}vy white noise.
\newblock {\em Electronic Journal of Probability}, 25:1--38, 2020.

\bibitem{aziznejad2020wavelet}
S.~Aziznejad and J.~Fageot.
\newblock The wavelet compressibility of compound {P}oisson processes.
\newblock {\em IEEE Transactions on Information Theory}, 68(4):2752--2766,
  2022.

\bibitem{barndorff2004levy}
O.E. Barndorff-Nielsen and J.~Schmiegel.
\newblock L{\'e}vy-based spatial-temporal modelling, with applications to
  turbulence.
\newblock {\em Russian Mathematical Surveys}, 59(1):65, 2004.

\bibitem{barron1986entropy}
A.R. Barron.
\newblock Entropy and the central limit theorem.
\newblock {\em Annals of Probability}, 14(1):336--342, 1986.

\bibitem{berger2019levydriven2}
D.~Berger.
\newblock L\'evy driven linear and semilinear stochastic partial differential
  equations.
\newblock {\em arXiv preprint arXiv:1907.01926}, 2019.

\bibitem{Bertoin1998levy}
J.~Bertoin.
\newblock {\em L{\'e}vy Processes}, volume 121.
\newblock Cambridge University Press, 1998.

\bibitem{Blumenthal1961sample}
R.M. Blumenthal and R.K. Getoor.
\newblock Sample functions of stochastic processes with stationary independent
  increments.
\newblock {\em Journal of Mathematics and Mechanics}, 10:493--516, 1961.

\bibitem{bobkov2019local}
S.G. Bobkov.
\newblock Local limit theorems for densities in {O}rlicz spaces.
\newblock {\em Journal of Mathematical Sciences (United States)},
  242(1):52--68, 2019.

\bibitem{bobkov2011convergence}
S.G. Bobkov, G.P. Chistyakov, and F.~G{\"o}tze.
\newblock Convergence to stable laws in relative entropy.
\newblock {\em Journal of Theoretical Probability}, pages 1--16, 2011.

\bibitem{Bottcher2014levy}
B.~B{\"o}ttcher, R.L. Schilling, and J.~Wang.
\newblock {\em L{\'e}vy {M}atters III: L{\'e}vy-{T}ype {P}rocesses:
  {C}onstruction, {A}pproximation and {S}ample {P}ath {P}roperties}, volume
  2099.
\newblock Springer, 2014.

\bibitem{Brockwell2009existence}
P.J. Brockwell and A.~Lindner.
\newblock Existence and uniqueness of stationary {L}{\'e}vy-driven {CARMA}
  processes.
\newblock {\em Stochastic Processes and Their Applications}, 119(8):2660--2681,
  2009.

\bibitem{Candes2006sparse}
E.J. Cand{\`e}s, J.~Romberg, and T.~Tao.
\newblock Robust uncertainty principles: Exact signal reconstruction from
  highly incomplete frequency information.
\newblock {\em IEEE Transactions on Information Theory}, 52(2):489--509, 2006.

\bibitem{Cevher2009learning}
V.~Cevher.
\newblock Learning with compressible priors.
\newblock In {\em Advances in Neural Information Processing Systems}, pages
  261--269, 2009.

\bibitem{Cover2012elements}
T.M. Cover and J.A. Thomas.
\newblock {\em Elements of information theory}.
\newblock John Wiley \& Sons, 2012.

\bibitem{Crooks2019field}
G.~E. Crooks.
\newblock {\em Field guide to continuous probability distributions}.
\newblock BITS, 1st edition, 2019.
\newblock \url{https://threeplusone.com/fieldguide}.

\bibitem{Dalang2015Levy}
R.C. Dalang and T.~Humeau.
\newblock L{\'e}vy processes and {L}{\'e}vy white noise as tempered
  distributions.
\newblock {\em Annals of Probability}, 45(6b):4389--4418, 2017.

\bibitem{Deng2015shift}
C.S. Deng and R.L. Schilling.
\newblock On shift {H}arnack inequalities for subordinate semigroups and moment
  estimates for {L}{\'e}vy processes.
\newblock {\em Stochastic Processes and Their Applications}, 125:3851--3878,
  2015.

\bibitem{Donoho2006}
D.L. Donoho.
\newblock Compressed sensing.
\newblock {\em IEEE Transactions on Information Theory}, 52(4):1289--1306,
  2006.

\bibitem{Durand2012multifractal}
A.~Durand and S.~Jaffard.
\newblock Multifractal analysis of {L}{\'e}vy fields.
\newblock {\em Probability Theory and Related Fields}, 153(1-2):45--96, 2012.

\bibitem{Embrechts2000introduction}
P.~Embrechts and M.~Maejima.
\newblock An introduction to the theory of self-similar stochastic processes.
\newblock {\em International Journal of Modern Physics B}, 14:1399--1420, 2000.

\bibitem{FageotThese}
J.~Fageot.
\newblock {\em Gaussian versus Sparse Stochastic Processes: Construction,
  Regularity, Compressibility}.
\newblock PhD thesis, EPFL, 2017.

\bibitem{Fageot2014}
J.~Fageot, A.~Amini, and M.~Unser.
\newblock On the continuity of characteristic functionals and sparse stochastic
  modeling.
\newblock {\em Journal of Fourier Analysis and Applications}, 20:1179--1211,
  2014.

\bibitem{Fageot2015wavelet}
J.~Fageot, E.~Bostan, and M.~Unser.
\newblock Wavelet statistics of sparse and self-similar images.
\newblock {\em SIAM Journal on Imaging Sciences}, 8(4):2951--2975, 2015.

\bibitem{Fageot2017multidimensional}
J.~Fageot, A.~Fallah, and M.~Unser.
\newblock Multidimensional {L}{\'e}vy white noise in weighted {B}esov spaces.
\newblock {\em Stochastic Processes and Their Applications}, 127(5):1599--1621,
  2017.

\bibitem{fageot2019scaling}
J.~Fageot and M.~Unser.
\newblock Scaling limits of solutions of linear stochastic differential
  equations driven by {L}{\'e}vy white noises.
\newblock {\em Journal of Theoretical Probability}, 32(3):1166--1189, 2019.

\bibitem{Fageot2017besov}
J.~Fageot, M.~Unser, and J.P. Ward.
\newblock On the {B}esov regularity of periodic {L}{\'e}vy noises.
\newblock {\em Applied and Computational Harmonic Analysis}, 42(1):21 -- 36,
  2017.

\bibitem{fageot2020nterm}
J.~Fageot, M.~Unser, and J.P. Ward.
\newblock The $n$-term approximation of periodic generalized {L}{\'e}vy
  processes.
\newblock {\em Journal of Theoretical Probability}, 33(1):180--200, 2020.

\bibitem{Farkas2001function}
W.~Farkas, N.~Jacob, and R.L. Schilling.
\newblock Function spaces related to continuous negative definite functions:
  {$\psi$}-{B}essel potential spaces.
\newblock {\em Dissertationes Math. (Rozprawy Mat.)}, 393:1--62, 2001.

\bibitem{Feller2008introduction}
W.~Feller.
\newblock {\em An introduction to probability theory and its applications},
  volume~2.
\newblock John Wiley \& Sons, 2008.

\bibitem{de1929sulle}
B.~De Finetti.
\newblock {\em Sulle funzioni a incremento aleatorio}.
\newblock Bardi, 1929.

\bibitem{Ghourchian2017compressible}
H.~Ghourchian, A.~Amini, and A.~Gohari.
\newblock How compressible are innovation processes?
\newblock {\em IEEE Transactions on Information Theory}, 64(7):4843--4871,
  2018.

\bibitem{Ghourchian2017existence}
H.~{Ghourchian}, A.~{Gohari}, and A.~{Amini}.
\newblock Existence and continuity of differential entropy for a class of
  distributions.
\newblock {\em IEEE Communications Letters}, 21(7):1469--1472, July 2017.

\bibitem{gnedenko1954limit}
B.V. Gnedenko and A.N. Kolmogorov.
\newblock {\em Limit distributions for sums of independent random variables}.
\newblock Addison-Wesley, Reading, MA, 1954.
\newblock Translated from the Russian edition (1949) by K.L. Chung.

\bibitem{grabchak2016tempered}
M.~Grabchak.
\newblock {\em Tempered Stable Distributions: Stochastic Models for Multiscale
  Processess}.
\newblock Springer, 2016.

\bibitem{Gribonval2012compressibility}
R.~Gribonval, V.~Cevher, and M.E. Davies.
\newblock Compressible distributions for high-dimensional statistics.
\newblock {\em {IEEE} Transactions on Information Theory}, 58(8):5016--5034,
  2012.

\bibitem{hastie2015statistical}
T.~Hastie, R.~Tibshirani, and M.~Wainwright.
\newblock {\em Statistical learning with sparsity: the lasso and
  generalizations}.
\newblock Chapman and Hall/CRC, 2015.

\bibitem{Houdre2007layered}
C.~Houdr{\'e} and R.~Kawai.
\newblock On layered stable processes.
\newblock {\em Bernoulli}, 13(1):252--278, 2007.

\bibitem{hsu1951absolute}
P.L. Hsu.
\newblock Absolute moments and characteristic functions.
\newblock {\em J. Chinese Math. Soc}, 1:259--280, 1951.

\bibitem{humphries2010environmental}
N.E. Humphries, N.~Queiroz, J.~Dyer, N.G. Pade, M.K. Musyl, K.M. Schaefer, D.W.
  Fuller, J.M. Brunnschweiler, T.K. Doyle, and J.~Houghton.
\newblock Environmental context explains {L}{\'e}vy and {B}rownian movement
  patterns of marine predators.
\newblock {\em Nature}, 465(7301):1066--1069, 2010.

\bibitem{Ihara93information}
S.~Ihara.
\newblock {\em Information Theory for Continuous Systems}.
\newblock World Scientific, 1993.

\bibitem{Jaffard1999multifractal}
S.~Jaffard.
\newblock The multifractal nature of {L}{\'e}vy processes.
\newblock {\em Probability Theory and Related Fields}, 114(2):207--227, 1999.

\bibitem{jaynes1957information}
E.T. Jaynes.
\newblock Information theory and statistical mechanics.
\newblock {\em Physical review}, 106(4):620, 1957.

\bibitem{Koltz2001laplace}
S.~Koltz, T.J. Kozubowski, and K.~Podgorski.
\newblock {\em The Laplace Distribution and Generalizations}.
\newblock Birkh{\"a}user, 2001.

\bibitem{Kuhn2017existence}
F.~K{\"u}hn.
\newblock Existence and estimates of moments for {L}{\'e}vy-type processes.
\newblock {\em Stochastic Processes and their Applications}, 127(3):1018--1041,
  2017.

\bibitem{kuhn2017levy}
F.~K{\"u}hn.
\newblock {\em L{\'e}vy {M}atters VI: L{\'e}vy-Type Processes: Moments,
  Construction and Heat Kernel Estimates}, volume 2187.
\newblock Springer, 2017.

\bibitem{levy1954theorie}
P.~L{\'e}vy.
\newblock {\em Th{\'e}orie de l'addition des variables al{\'e}atoires},
  volume~1.
\newblock Gauthier-Villars, 1954.

\bibitem{Luschgy2008moment}
H.~Luschgy and G.~Pag{\`e}s.
\newblock Moment estimates for {L}{\'e}vy processes.
\newblock {\em Electronic Communications in Probability}, 13:422--434, 2008.

\bibitem{Mainardi2008origin}
F.~Mainardi and S.~Rogosin.
\newblock The origin of infinitely divisible distributions: from de {F}inetti's
  problem to {L}{\'e}vy-{K}hintchine formula.
\newblock {\em arXiv preprint arXiv:0801.1910}, 2008.

\bibitem{Mallat1999}
S.~Mallat.
\newblock {\em A Wavelet Tour of Signal Processing. The Sparse Way}.
\newblock Academic Press, Boston, third edition, 2009.
\newblock With contributions from G. Peyr\'{e}.

\bibitem{Mandelbrot1982fractal}
B.B. Mandelbrot.
\newblock {\em The {F}ractal {G}eometry of {N}ature}.
\newblock W. H. Freeman and Co., San Francisco, Californie, 1982.

\bibitem{Mandelbrot1968}
B.B. Mandelbrot and J.W.~Van Ness.
\newblock Fractional {B}rownian motions, fractional noises and applications.
\newblock {\em SIAM {R}eview}, 10(4):422--437, 1968.

\bibitem{Mumford2010pattern}
D.~Mumford and A.~Desolneux.
\newblock {\em Pattern {T}heory: the {S}tochastic {A}nalysis of {R}eal-{W}orld
  {S}ignals}.
\newblock A.K. Peters, Ltd., Natick, MA, 2010.

\bibitem{Pesquet2002stochastic}
B.~Pesquet-Popescu and J.~L{\'e}vy V{\'e}hel.
\newblock Stochastic fractal models for image processing.
\newblock {\em Signal Processing Magazine, IEEE}, 19(5):48--62, 2002.

\bibitem{Pinelis2020existence}
I.~Pinelis.
\newblock Existence of the differential entropy for infinitely divisible laws,
  July 2020.
\newblock Mathoverflow, \url{https://mathoverflow.net/a/365657} (retrieved on
  2020-08-01).

\bibitem{prokhorov1952local}
Y.~V. Prokhorov.
\newblock A local theorem for densities.
\newblock {\em Akad. Nauk SSSR}, 83:797--800, 1952.

\bibitem{renyi1959dimension}
A.~R{\'e}nyi.
\newblock On the dimension and entropy of probability distributions.
\newblock {\em Acta Mathematica Hungarica}, 10(1-2):193--215, 1959.

\bibitem{Rioul2011information}
O.~Rioul.
\newblock Information theoretic proofs of entropy power inequalities.
\newblock {\em IEEE Transactions on Information Theory}, 57(1):33--55, January
  2011.

\bibitem{Rosenbaum2009first}
M.~Rosenbaum.
\newblock First order $p$-variations and {B}esov spaces.
\newblock {\em Statistics \& Probability Letters}, 79(1):55--62, 2009.

\bibitem{Rosinski2007tempering}
J.~Rosinski.
\newblock Tempering stable processes.
\newblock {\em Stochastic processes and Their applications}, 117(6):677--707,
  2007.

\bibitem{Taqqu1994stable}
G.~Samorodnitsky and M.S. Taqqu.
\newblock {\em Stable {N}on-{Gaussian} {P}rocesses: {S}tochastic {M}odels with
  {I}nfinite {V}ariance}.
\newblock Stochastic Modeling. Chapman \& Hall, New York, 1994.

\bibitem{Sato1994levy}
K.~Sato.
\newblock {\em L\'evy {P}rocesses and {I}nfinitely {D}ivisible
  {D}istributions}, volume~68.
\newblock Cambridge University Press, Cambridge, 2013.

\bibitem{Schilling1998growth}
R.L. Schilling.
\newblock Growth and {H}{\"o}lder conditions for the sample paths of {F}eller
  processes.
\newblock {\em Probability Theory and Related Fields}, 112(4):565--611, 1998.

\bibitem{Schilling2000function}
R.L. Schilling.
\newblock Function spaces as path spaces of {F}eller processes.
\newblock {\em Mathematische Nachrichten}, 217(1):147--174, 2000.

\bibitem{schilling2012bernstein}
R.L. Schilling, R.~Song, and Z.~Vondracek.
\newblock {\em Bernstein functions: theory and applications}, volume~37.
\newblock Walter de Gruyter, 2012.

\bibitem{schoutens2003levy}
W.~Schoutens.
\newblock {\em Lévy processes in finance: pricing financial derivatives}.
\newblock Wiley Series in Probability and Statistics. Wiley, 2003.

\bibitem{shannon2001mathematical}
C.E. Shannon.
\newblock A mathematical theory of communication.
\newblock {\em Bell System Technical Journal}, 27(3):379--423, July 1948.

\bibitem{shannon2001mathematical2}
C.E. Shannon.
\newblock A mathematical theory of communication.
\newblock {\em Bell System Technical Journal}, 27(4):623--656, October 1948.

\bibitem{silva2022compressibility}
J.F. Silva.
\newblock Compressibility analysis of asymptotically mean stationary processes.
\newblock {\em Applied and Computational Harmonic Analysis}, 56:61--97, 2022.

\bibitem{Silva2015characterization}
J.F. Silva and M.S. Derpich.
\newblock On the characterization of $\ell_p$-compressible ergodic sequences.
\newblock {\em IEEE Transactions on Signal Processing}, 63(11):2915--2928,
  2015.

\bibitem{Silva2012compressibility}
J.F. Silva and E.~Pavez.
\newblock Compressibility of infinite sequences and its interplay with
  compressed sensing recovery.
\newblock In {\em Signal \& Information Processing Association Annual Summit
  and Conference (APSIPA ASC), 2012 Asia-Pacific}, pages 1--5. IEEE, 2012.

\bibitem{Sinai1976self}
Y.G. Sinai.
\newblock Self-similar probability distributions.
\newblock {\em Theory of Probability \& Its Applications}, 21(1):64--80, 1976.

\bibitem{Srivastava2003advances}
A.~Srivastava, A.B. Lee, E.P. Simoncelli, and S.-C. Zhu.
\newblock On advances in statistical modeling of natural images.
\newblock {\em Journal of Mathematical Imaging and Vision}, 18(1):17--33, 2003.

\bibitem{tibshirani1996regression}
R.~Tibshirani.
\newblock Regression shrinkage and selection via the lasso.
\newblock {\em Journal of the Royal Statistical Society: Series B
  (Methodological)}, 58(1):267--288, 1996.

\bibitem{Unser2014sparse}
M.~Unser and P.~D. Tafti.
\newblock {\em An Introduction to Sparse Stochastic Processes}.
\newblock Cambridge University Press, 2014.

\bibitem{ushakov2011selected}
N.G. Ushakov.
\newblock {\em Selected topics in characteristic functions}.
\newblock Walter de Gruyter, 2011.

\bibitem{Verdu2006simple}
S.~{Verdu} and D.~{Guo}.
\newblock A simple proof of the entropy-power inequality.
\newblock {\em IEEE Transactions on Information Theory}, 52(5):2165--2166,
  2006.

\bibitem{Vetterli2002FRI}
M.~Vetterli, P.~Marziliano, and T.~Blu.
\newblock Sampling signals with finite rate of innovation.
\newblock {\em {IEEE} Transactions on Signal Processing}, 50(6):1417--1428,
  2002.

\bibitem{williams1991probability}
D.~Williams.
\newblock {\em Probability with martingales}.
\newblock Cambridge university press, 1991.

\end{thebibliography}

\newpage
\appendix

\section{Deferred Proofs from Section~\ref{sec:defineentropy}}
\label{app:existenceentropy}

\begin{proof}[Proof of \cref{prop:entropy-criterion}]
As already discussed above \cref{prop:entropy-criterion}, for an absolutely continuous random variable $X$, the condition $\E[\log(1+\abs X)]<\infty$ implies that $h(X)<\infty$. This is a direct application of Gibbs' inequality applied to $\KL(X\|Y)$ where $Y$ has a Cauchy distribution, as detailed in \cite[Proposition 1]{Rioul2011information}. Furthermore, if the probability density function $p_X$ of $X$ is bounded by some constant $a$, then $h(X)>\log 1/a$. Indeed, in this case
\begin{displaymath}
	-h(X) = \int_\R p_X(x)\log p_X(x)\rd x \leq \int_\R p_X(x)\log a\rd x = \log a\,.
\end{displaymath}
Finally, by Fourier inversion \cite[XV.3 Theorem 3]{Feller2008introduction}, a random variable $X$ such that $\Phi_X$ is integrable admits a bounded (and continuous) probability density function. This concludes the proof of the second claim of the proposition.

We now adapt \cite[Proposition 1]{Rioul2011information} to the discrete case. For $k \in \Z$, define $p_n = \P ( \floor X = k)$ and $q_k = \frac{c}{1+ k^2}$ where $c = \sum_{k\in \Z}(1+k^2)^{-1}$ is such that $\sum_{k\in \Z} q_k = 1$. By Gibbs' inequality,
\begin{align}\label{eq:foo}
	0 &\leq \ent(\floor X) 
= \sum_{k \in \Z} p_k \log (1 / p_k ) 
\leq \sum_{k \in \Z} p_k \log (1 / q_k )
= \sum_{k \in \Z} p_k \prn[\big]{- \log c + \log ( 1 + k^2)} \nonumber\\
&\leq - \log c + \sum_{k \in \Z} p_k \log \prn[\big]{ ( 1 + |k| ) ^2 }
  = - \log c + 2  \E\bra[\big]{\log ( 1 + | \floor X |)}.
\end{align}
Finally, using $| \floor X | \leq |X| + 1 \leq 2 |X| + 1$, we upper bound the last term in \eqref{eq:foo}
\begin{displaymath}
	\E [ \log ( 1 + | \floor X |) ]  \leq \E [ \log ( 2 + 2| X |) ] )
	=  \log 2 +   \E[ \log (1+|X| ) ],
\end{displaymath}
which is finite  by assumption.
Therefore, $\ent(\floor X)  < \infty$, which concludes the proof.
\end{proof}

\begin{proof}[Proof of Proposition~\ref{prop:infiniteentropy}]
	The key argument is to construct a symmetric infinitely divisible and absolutely continuous random variable $X$ such that $p_X(x)  \geq  \frac{c}{|x| \log^2 |x|}$ for some constant $c>0$ and $|x|$ large enough and to consider the Lévy process $L$ whose law at time $t=1$ is the one of $X$. {We reproduce here for the sake of completeness the construction of $p_X$ proposed by Iosif Pinelis on Mathoverflow \cite{Pinelis2020existence}.}
	
	{Let $f(x) = \frac{\alpha \One_{x \geq e}}{x \log^2 x}$ where $0 < \alpha < \infty$ is such that $f$ is a pdf with $\int_{\R} f(x) \mathrm{d} x = 1$. For $n\geq 0$, we denote $f^{*n} = f * \cdots * f$ the $n$th fold convolution with the convention that $f^{*0} = \delta$. 
	We define the functions, which are easily shown to be pdf, 
	\begin{displaymath}
	    f_t(x)  = \mathrm{e}^{-t}  \sum_{n=0}^{\infty}
		\frac{t^n f^{*n}(x)}{n!} , \quad  g_t(x) = \frac{\mathrm{e}^{- \frac{x^2}{2t}} }{\sqrt{2\pi t}}, 
	\end{displaymath}
    and $p_t  = f_t * g_t$. Then, we observe that, for $s,t >0$, by expanding $(s+t)^n$,
    \begin{displaymath}
        f_{s+t} = \mathrm{e}^{-(s+t)}  \sum_{n=0}^{\infty}
	    \frac{(s+t)^n f^{*n}}{n!}
	    = \left( \mathrm{e}^{-s}  \sum_{n=0}^{\infty}
	    \frac{s^n f^{*n}}{n!}\right) * \left(  \mathrm{e}^{-t}  \sum_{n=0}^{\infty}
	    \frac{t^n f^{*n}}{n!} \right)= f_s * f_t
    \end{displaymath}
    We also have that $g_{s+t} = g_s * g_t $ and therefore $p_{s+t} = p_s * p_t $. This shows that
    there exists an infinitely divisible random variable $X$ whose pdf is $p_X = p_1$. Moreover, $p_1$ is absolutely continuous as the convolution of a Gaussian pdf. Then, we have for $x \geq 1 + \mathrm{e}$ that
    \begin{displaymath}
        p_X(x) 
     = (f_1 * g_1)(x) \geq \mathrm{e}^{-1} (f * g_1)(x) 
        \geq \mathrm{e}^{-1}\int_{0}^1 f(x-y) g_1(y) \mathrm{d} y \geq \mathrm{e}^{-1} f(x) \int_0^1 g_1(y) \mathrm{d} y,
    \end{displaymath}
where we used that $f_1 \geq \mathrm{e}^{-1}f$ in the first inequality and that $f$ is decreasing on $(\mathrm{e}, \infty)$ for the last one. Finally, the choice of $f$ implies that $p_X(x) \geq \frac{c}{x \log^2 x}$ for some constant $0< c < \infty$ and every $x \geq \mathrm{e} + 1$, as desired.}
	
  Recall that $p_k = \mathbb{P} ( \floor{X} =k) = \int_{k}^{k+1} p_X(x)\mathrm{d}x$.
 Moreover, for $k \geq 1$,  we have
 \begin{displaymath}
 \int_k^{k+1} \frac{\mathrm{d}x}{x\log^2 x} = \frac{1}{\log k} - \frac{1}{\log (k+1)}=      \frac{\log (1 + 1 /k)}{\log k \log (k+1)}   \underset{k\rightarrow \infty}{\sim} \frac{1}{k \log^2 k}
 \end{displaymath}
 Hence, using that $p_X(x) \geq c / (x \log^2 x)$ for $x$ large enough, we deduce that, for $k$ large enough,
  \begin{displaymath}
     p_k \geq \frac{c/2}{k\log^2 k}.
  \end{displaymath}
 Using that $x \mapsto - x \log x$ is increasing on $(0,1/\mathrm{e}]$, we therefore deduce that, for $k$ large enough (in particular such that $p_k \leq 1/\mathrm{e}$),
  \begin{displaymath}
     -  p_k \log(p_k) \geq   \frac{c/2}{k \log^2 k} \log\left( \frac{k\log^2 k}{c/2} \right) \underset{k\rightarrow \infty}{\sim} \frac{c/2}{  k \log k}.
  \end{displaymath}
  This shows that $\ent(\floor{L_1}) = -  \sum_{k \in \Z} p_k \log (p_k)  = \infty$ since $\sum_{k\geq 2} \frac{1}{k\log k} = \infty$. 
\end{proof}

\section{Deferred Proofs from Section~\ref{sec:betanonzero}}
\label{app:deferred-betanonzero}

\begin{proof}[Proof of Proposition~\ref{prop:stableandlocstable}]
    The characterization of self-similar Lévy processes is classic and proved for instance in~\cite[Proposition 13.5]{Sato1994levy}. 

    Assume that the characteristic exponent of $L$ satisfies \eqref{eq:psilimitforcarac}. Then, it is LSS according to \cite[Proposition 5.8]{fageot2019scaling} (applied to $d=1$). The local self-similarity order is then given by $1/\beta$ according to \cite[Theorem 4.6]{fageot2019scaling}. 
    For the converse, assume that $L$ is a  LSS Lévy process. First, this implies that $\beta > 0$ due to~\cite[Proposition 4.7]{fageot2019scaling}.
    Let $Y$ be the limit in law of {$b^H L_{\cdot / b}$ when $b\rightarrow \infty$}, which is symmetric by assumption. Denoting by $\overset{(\mathcal{L})} {\longrightarrow}$ the convergence in law,
    we have for all $a >0$ that
    \begin{displaymath}
     (ab)^H L_{\cdot / ab} \overset{(\mathcal{L})} {\underset{b\rightarrow \infty}{\longrightarrow}} Y   .  
    \end{displaymath} 
    Moreover, {$Y_{\cdot / a}$ is the limit in law of $b^H L_{\cdot / ab}$ when $b \rightarrow \infty$ and we therefore have that}
    \begin{displaymath}
 (ab)^H L_{\cdot / ab} = {a^H \left( b^H L_{\cdot / ab} \right)} \overset{(\mathcal{L})} {\underset{b\rightarrow \infty}{\longrightarrow}}
   a^H Y_{\cdot / a}   .
    \end{displaymath} 
    The unicity of the limit implies that, for all $a > 0$, $a^H Y_{\cdot / a}$ and $Y$ are equal in law, hence $Y$ is self-similar of order $H$.
    Moreover, $Y$ is a Lévy process: limiting arguments show that $Y_0 = 0$ a.s., and that its increments are stationary and independent.
    The first part of Proposition~\ref{prop:stableandlocstable} therefore implies that the symmetric and self-similar Lévy process $Y$ is S$\alpha$S with $\alpha = 1/ H$.
	Its characteristic exponent is therefore given by $\Psi_Y(\xi) =  - \gamma \abs \xi ^{1/H}$ for some $\gamma >0$ and all $\xi\in \R$. 
    Let $\Psi$ be the characteristic exponent of $L$. Then, the convergence in law of $a^H L_{\cdot / a}$ to $Y$ implies in particular the pointwise convergence of the characteristic functions of $a^H L_{1/a}$ to the one of $Y_1$ when $a\rightarrow \infty$. This implies that
    \begin{equation} \label{eq:limitminusgamma}
        \frac{1}{a} \Psi(a^H) = \log \mathbb{E} [\mathrm{e}^{\mathrm{i} a^H L_{1/a}} ]
        \underset{a\rightarrow \infty}{\longrightarrow} \log \mathbb{E} [\mathrm{e}^{\mathrm{i} Y_1} ] = \Psi_Y(1)  = -\gamma.
    \end{equation}
    Setting $a^H = \xi$ in \eqref{eq:limitminusgamma}, we deduce that $\Psi(\xi)  {\sim} -\gamma \xi^{1/H}$  when $\xi \rightarrow   \infty$.
	We show similarly that $\Psi(\xi)  {\sim} -\gamma \abs\xi^{1/H}$ when $\xi \rightarrow - \infty$.
    Finally, using again \cite[Theorem 4.6]{fageot2019scaling}, we know that $1/H = \beta$ is the Blumenthal--Getoor index of $L$ and Proposition~\ref{prop:stableandlocstable} is proved.
    \end{proof}

	\begin{proof}[Proof of \cref{lemma:controlmoments}]
	Define $r=\min\set{1, p,\beta}$. Since $\E[|L_1|^p]<\infty$ by assumption, we also have $\E[|L_1|^r]<\infty$. By \cite[Thm. 25.3]{Sato1994levy} applied to $g(t) = \max (|t|^r, 1)$, this is equivalent to $\int_{\abs t\geq 1} \abs t^r \nu(\rd t)<\infty$.
	 We can therefore apply \cite[Prop. 2.4]{FageotThese} with $r \leq 1$ and deduce that $\abs{\Psi(\xi)}\leq C_1 \abs\xi^r$ for some $C_1 > 0$ and any $\abs\xi<1$.
	Moreover, $\abs{\Psi(\xi)} \leq C_2 \abs\xi^\beta$ for some $C_2 > 0$ and any $\abs\xi\geq  1$ by assumption. 
    Then,  for all $\xi\in\R$, we have that 
	\begin{align} \label{eq:atzerpsi}
 \abs{\Psi(\xi)} \leq C\prn[\big]{\abs\xi^\beta + \abs\xi^r}
 \end{align}
 with $C = \max(C_1,C_2)$.
 
 For all $q \in (0,2)$ and random variable $X$ such that $\E[|X|^q] < \infty$, we have the classic\footnote{This condition appears apparently for the first time in the study of fractional absolute moments of random variables by Pao-Lu Hsu in~\cite{hsu1951absolute}.} expression of $\E[|X|^q]$ in terms of the characteristic function $\Phi_X$ (see \emph{e.g.}, ~\cite[Thm 1.5.9]{ushakov2011selected}):
\begin{equation} \label{eq:moments1}
    \E[|X|^q] = c_q \int_{\R} \frac{1 -  \Re \Phi_X(\xi)}{|\xi|^{q+1}} \rd\xi,
\end{equation}
where $c_q>0$ is a constant depending only on $q$. Moreover, if $X$ is infinitely divisible with characteristic exponent $\Psi$, there exists a constant $M>0$ such that for all $\xi\in\R$,
\begin{align}\label{eq:bound-char}
    1 - \Re \Phi_X(\xi) & = \Re ( 1 - \Phi_X(\xi) ) \leq |1 - \Phi_X(\xi) | \nonumber \\ &
    \leq M (1 - \exp( - |\Psi(\xi)|),
\end{align}
where the last inequality uses that {the function $f:z\mapsto |1-e^z|/(1-e^{-|z|})$ is bounded by some constant $M>0$ for all $z\in\mathbb{C}$ with $\Re z\leq 0$. Indeed, when $|z|\leq 1$ we use the inequality $|1-e^z|\leq e^{|z|} - 1$ (easily proved using the power series expansion of $e^z$ and the triangle inequality) which implies $f(z)\leq \frac{e^{|z|}-1}{1-e^{-|z|}}=e^{|z|}\leq e$. When $|z|\geq 1$, we use $|1-e^z|\leq 1 + |e^z| = 1+e^{\Re z}\leq 2$ for $z$ with $\Re z\leq 0$, which implies $f(z)\leq \frac{2}{1-e^{-|z|}}\leq \frac{2e}{e-1}$.}

{The bound \eqref{eq:bound-char} together with \eqref{eq:moments1}} implies
\begin{equation} \label{eq:moments}
    \E[|X|^q] \leq c_q M  \int_{\R} \frac{1 -  \exp(-|\Psi(\xi)|)}{|\xi|^{q+1}} \rd\xi.
\end{equation}

For $X=t^{-1/\beta} L_t$ whose characteristic exponent is $\xi \mapsto  t \Psi( t^{-1/\beta} \xi)$ we have by \eqref{eq:atzerpsi}
\begin{displaymath}
	t\Psi(t^{-1/\beta}\xi)\leq C\big(\abs\xi^\beta+ t^{1-r/\beta}\abs\xi^r\big)
	\leq C\big(\abs\xi^\beta + \abs\xi^r\big)
\end{displaymath}
for all $0<t\leq 1$, where the second inequality uses that $t^{1-r/\beta}\leq 1$ since $r\leq\beta$. Using that $y\mapsto 1-\exp(-y)$ is increasing, we get from \eqref{eq:moments} that for all $0<t\leq 1$,
\begin{equation}\label{eq:finalestimate}
\E[|X|^q] \leq c_q M  \int_{\R} \frac{1 -  \e^{-C(\abs\xi^\beta + \abs\xi^r)}}{|\xi|^{q+1}} \rd\xi.
\end{equation}
Finally, for $0<q<r$, the integral in \eqref{eq:finalestimate} is finite since the integrand is equivalent to $1/\abs\xi^{q+1}$ when $\abs\xi\to\infty$, and to $C\abs\xi^{r-q-1}$ with $r > q$, when $\xi\to 0$. Thus, $\E[|X|^q]$ is upper-bounded by a constant independent of $t$ that we denote by $M_q$. This implies \eqref{eq:boundMpforLt}.
\end{proof}

\section{Deferred Proof from Section~\ref{sec:examples}}
\label{app:deferred-examples}

  \begin{proof}[Proof of Proposition~\ref{prop:temperedstable}]
	{The Blumenthal--Getoor index $\beta$ of $L$ is characterized in~\cite[Lemmas 3.22 and 3.23]{grabchak2016tempered} and $\beta = \alpha$.
    We study the asymptotic behavior of the characteristic exponent $\Psi$. Tempered stable processes being symmetric, $\Psi$ is  given by \eqref{eq:caractexposimple}. As we did for layered stable processes, we  use  the decomposition  $\Psi = \Psi_0 + \Psi_\infty$  with
    \begin{displaymath}
    \Psi_0(\xi) = \int_{|t|\le 1}  (\cos ( t \xi ) - 1 )  \nu( \mathrm{d} t ) 
	\quad\text{and}\quad
    \Psi_\infty(\xi) = \int_{|t|> 1}  (\cos ( t \xi ) - 1 )  \nu( \mathrm{d} t ).
    \end{displaymath}
    We then observe that, by a simple change of variable and using \eqref{eq:nutempered}, 
    \begin{equation} \label{eq:psi0new}
     \Psi_0(\xi) = - 
	 \left( \int_{\R} (1 -\cos u) \frac{q(  |u / \xi| ^p )}{\abs u^{\alpha+1}}  \One_{\abs u \leq \abs \xi} \mathrm{d}u \right) \abs \xi^{\alpha} 
	 := - \left( \int_{\R} h_\xi (u) \mathrm{d} u \right) \abs \xi^{\alpha} . 
    \end{equation}
Then, we have that $h_{\xi}(u) \rightarrow \frac{c(1-\cos u)}{|u|^{\alpha + 1}}$ when $|\xi|\rightarrow \infty$ for any $u \in \R$ where $c > 0$ is given by \eqref{eq:tructruc}. Moreover, the function $q$ being continuous, it is in particular bounded over $[0,1]$, let say by $M > 0$. Then, we have that  
$$h_\xi (u) \leq q\left( \left| \frac{u} {\xi} \right|^p \right) \One_{|u| \leq |\xi|} \frac{ 1 - \cos u }{ |u|^{\alpha + 1}}  \leq  M \frac{ 1 - \cos u }{ |u|^{\alpha + 1}},
$$
which is integrable over $\R$. The Lebesgue dominated convergence theorem therefore applies and $\int_{\R} h_\xi (u) \mathrm{d} u$ converges to $\gamma = \int_{\R} \frac{c(1-\cos u)}{|u|^{\alpha + 1}} \mathrm{d} u > 0$. Hence, \eqref{eq:psi0new} implies that $\Psi_0(\xi) \sim_{|\xi|\rightarrow \infty}  - \gamma |\xi|^{\alpha}$. }

{ Moreover, the function $q$ is continuous and vanishes at infinity. It is therefore bounded and we deduce that
    \begin{displaymath}
		\lvert \Psi_\infty(\xi) \rvert =  \int_{|t|> 1}  \big(1 - \cos (\xi t)\big) \frac{q(|t|^p)}{|t|^{\alpha + 1}} \mathrm{d} t 
		 \leq 2 \norm q _\infty \int_{|t|> 1}  \frac{\mathrm{d} t}{|t|^{\alpha + 1}}, 
    \end{displaymath}
	which is a constant independent from $\xi$. }
	
    {We have therefore shown that $\Psi(\xi) \sim_{|\xi|\rightarrow \infty}  - \gamma |\xi|^{\alpha}$, hence $L$ is LSS with of order $H = 1 / \alpha$ by \cref{prop:stableandlocstable}. Moreover, for $r < \alpha$, we have that 
	\begin{displaymath}
	    \int_{|t| > 1} |t|^r \nu(\mathrm{d} t) 
	     = \int_{|t|> 1} \frac{q(|t|^p)}{|t|^{1 + (\alpha - r)}} \mathrm{d}t 
	     \leq 
		\norm q _\infty \int_{|t|> 1} \frac{\mathrm{d} t}{|t|^{1 + (\alpha - r)}} < \infty.
	\end{displaymath}
	This implies that $\mathbb{E} [|L_1|^r] < \infty$ for some $r > 0$ and the conditions of Theorem~\ref{theo:maintheolocallyselfsim} apply from which we obtain \eqref{eq:entropy-layered}.}
\end{proof}

\end{document}